\documentclass[11pt]{amsart}
\usepackage[utf8]{inputenc}
\usepackage[english]{babel}
\usepackage{amsmath}
\usepackage{amsthm}
\usepackage{amssymb}
\usepackage{mathtools}
\usepackage{tikz-cd}
\usepackage[T1]{fontenc}
\usepackage{wasysym}
\usepackage{dsfont}
\usepackage[many]{tcolorbox}
\usepackage{soul}
\usepackage{xcolor}
\usepackage{appendix}
\usepackage{enumerate}
\usepackage{multicol}
\usepackage{fancyhdr}
\usepackage{parskip}
\usepackage{hyperref}
\usepackage{mathtools}
\usepackage{nameref}
\usepackage{thmtools}
\usepackage{thm-restate}
\usepackage{adjustbox}
\usepackage{cleveref}
\usepackage{tikz}
\usepackage{float}
\usepackage{MnSymbol}
\usepackage{amsthm}
\usepackage{orcidlink}
\usetikzlibrary{graphs.standard} 

\hypersetup{
  colorlinks   = true, 
  urlcolor     = {blue!50!black}, 
  linkcolor    = {blue!50!black}, 
  citecolor   = {red!50!black} 
}

\newcommand{\ar}{\mathrm{ar}}

\renewcommand{\L}{\mathcal{L}}

\newcommand{\obj}{\mathrm{obj}}
\newcommand{\morph}{\mathrm{morph}}
\newcommand{\Gra}{\mathbf{Gra}}
\newcommand{\Str}{\mathbf{Str}}
\newcommand{\Pos}{\mathbf{Pos}}
\newcommand{\StrL}{\Str(\L)}
\newcommand{\Sym}{\mathbf{Sym}}
\newcommand{\W}{\mathbf{W}}
\newcommand{\Syst}{\mathbf{Syst}}

\newcommand{\Hom}{\mathrm{Hom}}
\newcommand{\Gdg}{\mathbf{Gdg}}
\newcommand{\GdgL}{\Gdg(\L)}
\newcommand{\End}{\mathrm{End}}
\newcommand{\Alg}{\mathbf{Alg}}

\newcommand{\A}{\mathbb{A}}

\newcommand{\cf}{\mathrm{cf}}
\newcommand{\Ord}{\mathbf{Ord}}
\newcommand{\ZFC}{\mathsf{ZFC}}
\newcommand{\Card}{\mathbf{Card}}

\newcommand{\iso}{\simeq}
\newcommand{\C}{\mathcal{C}}

\newtheorem{theorem}{Theorem}[section]
\newtheorem*{theorem*}{Theorem}
\newtheorem{lemma}[theorem]{Lemma}
\newtheorem{proposition}[theorem]{Proposition}
\newtheorem*{proposition*}{Proposition}
\newtheorem{corollary}[theorem]{Corollary}
\newtheorem*{corollary*}{Corollary}
\newtheorem*{fact*}{Fact}
\newtheorem{claim}{Claim}[theorem]
\newenvironment{claimproof}[1]{\par\noindent\underline{Proof:}\space#1}{\hfill $\blacksquare$}
\newtheorem{fact}[theorem]{Fact}

\theoremstyle{definition}
\newtheorem*{definition*}{Definition}
\newtheorem{definition}[theorem]{Definition}

\newtheorem{question}[theorem]{Question}
\newtheorem*{question*}{Question}

\newtheorem*{notation*}{Notation}

\theoremstyle{remark}

\newcommand{\Ind}[1]
{#1\setbox0=\hbox{$#1x$}\kern\wd0\hbox to 0pt{\hss$#1\mid$\hss} \lower.9\ht0\hbox to 0pt{\hss$#1\smile$\hss}\kern\wd0}

\newcommand{\notind}[1]
{#1\setbox0=\hbox{$#1x$}\kern\wd0
\hbox to 0pt{\mathchardef\nn=12854\hss$#1\nn$\kern1.4\wd0\hss}
\hbox to 0pt{\hss$#1\mid$\hss}\lower.9\ht0 \hbox to 0pt{\hss$#1\smile$\hss}\kern\wd0}

\newcommand{\Gaif}{\mathsf{Gaif}}
\newcommand{\Arc}{\mathsf{Arc}}

\newcommand{\Cfrak}{\ensuremath{\mathfrak{C}}}
\newcommand{\Dfrak}{\ensuremath{\mathfrak{D}}}

\newcommand{\Gfrak}{\ensuremath{\mathfrak{G}}}

\newcommand{\Acal}{\ensuremath{\mathcal{A}}}

\newcommand{\Hcal}{\ensuremath{\mathcal{H}}}
\newcommand{\Ical}{\ensuremath{\mathcal{I}}} 
 
\newcommand{\Kcal}{\ensuremath{\mathcal{K}}} 

\newcommand{\Mcal}{\ensuremath{\mathcal{M}}}

\newcommand{\Rcal}{\ensuremath{\mathcal{R}}}
\newcommand{\Scal}{\ensuremath{\mathcal{S}}}

\newcommand{\Vcal}{\ensuremath{\mathcal{V}}}

\title{Algebraically universal categories of relational structures}

\author[I. Eleftheriadis]{Ioannis {Eleftheriadis}\ \orcidlink{0000-0003-4764-8894}}
\address{Ioannis {Eleftheriadis}, Department of Computer Science and Technology, University of Cambridge, UK}
\email{\href{ie257@cam.ac.uk}{ie257@cam.ac.uk}}

\thanks{Supported by a George and Marrie Vergottis Scholarship awarded through Cambridge Trust, an Onassis Foundation Scholarship, and a Robert Sansom Studentship.}


\begin{document}
\maketitle

\begin{abstract}
    We consider categories of relational structures that fully embed every category of universal algebras, and prove a partial characterisation of these in terms of an infinitary variant of the notion of nowhere density of Nešetřil and Ossona de Mendez. More precisely, we show that the Gaifman class of an algebraically universal category contains subdivided complete graphs of any infinite size, and establish that any monotone category satisfying this may be oriented to obtain an algebraically universal category. For the proof of the above, we also develop a categorical framework for relational gadget constructions. This generalises known results about categories of finite graphs to categories of relational structures of unbounded size. 
\end{abstract}

\section{Introduction}

In his seminal book on graph theory \cite{konig}, K{\"o}nig first proposed the problem of whether it is possible to represent a given abstract group as the automorphism group of some graph. This was answered affirmatively for finite groups by Frucht \cite{frucht}, and since then, there has been a series of results regarding the representation of groups in various finite or infinite structures \cite{birkhoff}, \cite{groot}, \cite{sabidussi}. These culminated in the work of Isbell \cite{isbell}, who proposed full embeddings as the means of extending these representation results to a general setting, and lead to the study of \emph{algebraically universal categories}, i.e.\ those categories that fully embed all categories of universal algebras.

Algebraically universal categories were extensively studied by the Prague school of category theory in the 1960s, leading to a number of important results summarised in \cite{pultr}. Amongst others, it was shown that algebraically universal categories fully embed all small categories \cite{hedrlin1966full}, i.e.\ those categories whose morphism class is a set, and so in particular any monoid can be realised as the endomorphism monoid of an object in an algebraically universal category. Moreover, it was established that the category $\Gra$ of all graphs with homomorphisms is algebraically universal \cite{pultr1964concerning}. In turn, this motivated investigation aiming to identify those full subcategories of $\Gra$ that are algebraically universal. 

More recently, a partial characterisation of these categories of graphs was established by Nešetřil and Ossona de Mendez in the framework of finite set theory \cite{nesetrilossona}. Surprisingly, this connects algebraic universality with a graph sparsity notion known as \emph{nowhere density}. Nowhere density was introduced by the same two authors \cite{nowhere}, \cite{NOdM12} as structural property of classes of finite graphs that generalises numerous well-behaved classes, including graphs of bounded degree, planar graphs, graphs excluding a fixed minor and graphs of bounded expansion. 

An important fact used in this characterisation of algebraic universality is that a monotone class of finite graphs is nowhere dense if, and only if, it is stable in the sense of model theory. This connection between stability and combinatorial sparsity was established in the context of infinite graphs by Podewski and Ziegler \cite{podewskiziegler}, and extended to classes of finite graphs by Adler and Adler \cite{AdlerAdler2014}. Answering a question of Adler and Adler, this picture was further extended to classes of arbitrary relational structures via their Gaifman graphs by Braunfeld, Dawar, Papadopoulos, and the author \cite{monotoneNIP}. 

Motivated by all the above, this paper generalises the characterisation of algebraically universal categories given by Nešetřil and Ossona de Mendez to categories of relational structures with homomorphisms, without the finiteness assumption present in \cite{nesetrilossona}. This is given in terms of an infinitary variant of nowhere density, and a generalisation of the arguments in \cite{monotoneNIP}. Our contribution can be summarized by the following theorem.

\begin{theorem}\label{th:maintheorem}
    Let $\Cfrak$ be an algebraically universal category of relational structures. Then $\Gaif(\Cfrak)$ is totally somewhere dense. Moreover, if $\Cfrak$ is a monotone category of relational structures such that $\Gaif(\Cfrak)$ is totally somewhere dense, then there is a full orientation $\tilde \Cfrak$ of $\Cfrak$ such that $\tilde{\Cfrak}$ is algebraically universal.
\end{theorem}

The various techniques used in the proof of the above are developed into general machinery that are possibly of broader interest. For instance, we define interpretable categories and show how full embeddings from these provide model-theoretic interpretations. Moreover, we develop a categorical framework for relational gadget constructions that extends the one for graphs in a non-trivial way; here, gadgets are glued in a way that permits different copies of the gadget to interact with one another. Finally, directed relational structures are defined, generalising the concept of a directed graph. These come with a natural faithful functor into the category of directed graphs. 

The necessary background and notation is established in Section~\ref{sec:prelim}. In Section \ref{sec:interpret} we introduce interpretable categories, and subsequently use them to establish that universal categories are somewhere dense in Section \ref{sec:swd}. The proof of the partial converse in the context of monotonicity occupies the next four sections. In Section~\ref{sec:gadgets}, we define relational gadgets and proceed to study their category under homomorphisms. With the gadget technique, we establish in Section \ref{sec:wellfdd} that the category of well-founded graphs with homomorphisms is algebraically universal. This, together with the directed structures introduced in Section \ref{sec:directed}, is instrumentally used in Section~\ref{sec:monotone} in the proof that totally somewhere dense monotone categories admit a full orientation which is algebraically universal.

 \section{Preliminaries}\label{sec:prelim}

    We work in Zermelo-Fraenkel set theory with the Axiom of Choice ($\ZFC$). Throughout this paper, $\L$ denotes a first-order relational language. We write $\mathrm{ar}(R)$ for the arity of each relation symbol $R \in \L$. For a category $\Cfrak$ we write $\obj(\Cfrak)$ for its class of objects and $\morph(\Cfrak)$ for its class of morphisms. For a cardinal $\kappa$ we write $\Cfrak_{<\kappa}$ for the full subcategory of $\Cfrak$ on the objects of size $<\kappa$. We sometimes abuse notation and write $M \in \Cfrak$ to mean that $M \in \obj(\Cfrak)$. Tuples of elements or variables are treated as functions, e.g.\ given a tuple $\bar a$ we write $\bar a(i)$ to denote the $i$-th element of $\bar a$. Likewise, functions from a cardinal $\lambda$ to a set $S$ are frequently treated as tuples of length $\lambda$ from $S$. For a finite cardinal $n \in \omega$ we write $[n]$ for the set $n+1\setminus 1=\{1,\dots,n\}$. 
    
\subsection{Graphs and relational structures}

An $\L$-structure is denoted by $(M,R^M)_{R \in \L}$, where $M$ is its underlying set and $R^M \subseteq M^{\ar(R)}$ is the interpretation of the relation symbol $R \in \L$ in $M$. By abusing notation, often we do not distinguish between an $\L$-structure and its underlying set. 

We say that $x \in M$ is an \emph{isolated point} of $M$ if there is no tuple $\bar m$ from $M$ such that $x \in \bar m$ and $\bar m \in R^M$ for some $R \in \L$. 

A \emph{homomorphism} from an $\L$-structure $M$ to an $\L$-structure $N$ is a map $f:M \to N$ such that for all relation symbols $R \in \L$ and tuples $\bar m \in M^{\ar(R)}$
\begin{equation*}\tag{$*$}\label{eq}
    \bar m \in R^M \implies f(\bar m) \in R^N.
\end{equation*}
A homomorphism $f: M \to N$ is said to be \emph{strong} whenever (\ref{eq}) is a bi-implication. We write $\StrL$ for the category of all $\L$-structures with homomorphisms. 

A formula $\phi(\bar x)$ is said to be primitive positive if it is equivalent to an existential formula $\exists \bar y \psi(\bar x,\bar y)$ where $\phi$ is a conjunction of atomic formulas. To every primitive positive formula $\phi(\bar x)$ (possibly in infinitary logic) we may associate a pointed $\L$-structure $(M_\phi, \bar x)$ whose domain is the set of variables of $\phi$, and for all $R \in \L$ $M_\phi\vDash R(v_1, \dots,v_n)$ if, and only if, $R(v_1,\dots,v_n)$ appears as a conjunct in $\psi(\bar x,\bar y)$. We call this \emph{the canonical structure of $\phi$}. It is easy to see that for any $\L$-structure $A$ and $\bar a \in A$ we have that $A \models \phi(\bar a)$ if, and only if, there is a homomorphism of pointed structures $h: (\Mcal_\phi,\bar x) \to (A, \bar a)$. 

We say that a class of $\L$-structures is \emph{monotone} if it is closed under inverse injective homomorphisms, i.e.\ if $B \in \C$, $A$ is an $\L$-structure and $f:A \to B$ is an injective homomorphism then $A \in \C$. Likewise, we say that $\C$ is \emph{hereditary} if it is closed under inverse injective strong homomorphisms. By a monotone (resp. hereditary) category $\Cfrak$ of $\L$-structures we mean one such that $\obj(\Cfrak)$ is monotone (resp. hereditary).

We say that two $\L$-structures $M,N$ are \emph{permutation equivalent} if $N$ can be obtained from $M$ by applying to every tuple $\bar m \in R^M$ a unique permutation $\sigma_{\bar m} \in S_{\ar(R)}$. More precisely, $M$ and $N$ are permutation equivalent whenever there is a bijection $f: M \to N$ such that for all $R \in \L$ and all $\bar m \in R^M$ there is a permutation $\sigma_{\bar m} \in \Scal_n$ satisfying $R^N=\{\sigma_{\bar m}(f(\bar m)) : \bar m \in R^M\}$. 

\begin{definition}
    Let $\C$ be a class of $\L$-structures. We say that a class $\tilde{\C}$ is an \emph{orientation} of $\C$ if for every $M \in \C$, $\tilde{\C}$ contains some $\tilde{M}$ that is permutation equivalent to $M$. For a subcategory $\Cfrak$ of $\StrL$, we say that $\tilde{\Cfrak}$ is a \emph{full orientation} of $\Cfrak$ if $\tilde{\Cfrak}$ is the full subcategory of $\StrL$ with $\obj(\tilde{\Cfrak})$ some orientation of $\obj(\Cfrak)$. 
\end{definition}

By a graph $G$ we mean an $\{E\}$-structure, where $E$ is a fixed binary relation symbol (so graphs here can have loops and symmetric directed edges). We write $E(G)$ rather than $E^G$ for the edge set of a graph. We say that a graph $G$ is \emph{undirected} if it $E(G)$ is non-reflexive and symmetric, and that it is \emph{directed} if $E(G)$ is non-symmetric and anti-symmetric. We write $\Gra$, $\mathbf{Sym}\Gra$, and $\overrightarrow{\Gra}$ for the categories of graphs, undirected graphs, and directed graphs respectively with graph homomorphisms. 

 For an undirected graph $G$ and $S \subseteq G$ we write $N^G(S)$ for the vertices in $G$ that are reachable by a path from some $s \in S$. Moreover, given an undirected graph $G$ and $r \in \omega$, we write $G^r$ for the \emph{$r$-subdivision of $G$}, i.e.\ the undirected graph obtained by replacing each edge of $G$ by an undirected path of length $r + 1$. Likewise, given a directed graph $G$ we write $G^{(r)}$ for the directed graph obtained by replacing each directed edge $(u,v)$ of $G$ by a path of length $r + 1$ directed from $u$ to $v$. We refer to vertices of $G$ present in $G^r$ ($G^{(r)}$ resp.) as \emph{native}, and to the remaining as \emph{subdivision points}. For a cardinal $\lambda$, we write $K_\lambda$ for the complete undirected graph on $\lambda$ vertices.

\begin{definition}
    Let $\C$ be a class of undirected graphs and $\kappa$ a cardinal. We say that $\C$ is \emph{$\kappa$-nowhere dense} if for every $r \in \omega$ there exists some cardinal $\lambda <\kappa$ such that $K^r_\lambda$ is not a subgraph of any $G \in \C$. Otherwise we say that $\C$ is \emph{$\kappa$-somewhere dense}. We also say that a class is \emph{eventually nowhere dense} if it is $\kappa$-nowhere dense for some cardinal $\kappa$. Conversely, we say that $\C$ is \emph{totally somewhere dense} if $\C$ is $\kappa$-somewhere dense for all cardinals $\kappa$. 
\end{definition}

It is easy to see that for cardinals $\lambda<\kappa$, $\lambda$-nowhere density implies $\kappa$-nowhere density. Moreover, the pigeonhole principle (see \Cref{fact:pigeon}) implies that a class $\C$ is totally somewhere dense if there is some $r \in \omega$ such that for all cardinals $\kappa$, the $r$-subdivided clique of size $\kappa$ is a (not necessarily induced) subgraph of some graph in $\C$.

In this context, nowhere density in the standard sense of Nešetřil and Ossona de Mendez \cite{nowhere} is precisely $\omega$-nowhere density. The connection between $\omega$-nowhere density and stability was first illustrated for graphs by Podewski and Ziegler \cite{podewskiziegler}. Recall that a class $\C$ of $\L$-structures is stable if for all $\L$-formulas $\phi(\bar x,\bar y)$ there is $n \in \omega$ such that there is no $M \in \C$ and $(\bar a_i)_{i \in n}$ from $M$ satisfying $M \models \phi(\bar a_i,\bar a_j) \iff i<j$. Going beyond graphs to relational structures requires working with the class of Gaifman graphs. 

\begin{definition}[Gaifman graph]
Given an $\L$-structure $(M,R^M)_{M \in \L}$ we define the \emph{Gaifman graph} (or \emph{underlying graph}) of $M$, denoted $\Gaif(M)$, to be the undirected graph on vertex set $M$ satisfying:
\begin{center}
    $(x,y) \in E(\Gaif(M)) \iff \exists R \in \L,  \exists v_1,\dots,v_{\mathrm{ar}(R)-2},$ $\exists  \sigma \in \Scal_{\mathrm{ar}(R)} \text{ such that } \sigma(x,y,v_1,\dots,v_{\mathrm{ar}(i)-2}) \in R^M. $
\end{center}
\end{definition}

Intuitively, this is formed by adding an edge between two elements of $M$ whenever they appear together in a relation. For a class of $\L$-structures $\mathcal{C}$ we define the \emph{Gaifman class} of $\C$ to be $\Gaif(\C):=\{ \Gaif(M) : M \in \C \}$. Likewise, given a category $\Cfrak$ of $\L$-structures we write $\Gaif(\Cfrak)$ for the class $\Gaif(\obj(\Cfrak))$ of undirected graphs. 

Phrased in the terminology used here, the following was established by Braunfeld, Dawar, Papadopoulos and the author, generalising the results of Podewski-Zigler and Adler-Adler. 

\begin{theorem}[\cite{monotoneNIP}]\label{monotoneNIP}
    Let $\C$ be a class of $\L$-structures such that $\Gaif(\C)$ is $\omega$-nowhere dense. Then $\C$ is stable. Moreover, if $\C$ is monotone and $\omega$-somewhere dense then $\C$ is unstable. 
\end{theorem}

\subsection{Ramsey Theory}\label{subs:ramsey}

The core idea of Ramsey theory is that if a structure is large enough then regularities in it are inevitable. This makes sense in both the finite and infinite contexts. In its simplest form, this idea gives the pigeonhole principle. 

\begin{fact}[Pigeonhole principle]\label{pigeonhole}
    Let $\kappa$ be a cardinal, $\lambda<\cf(\kappa)$ and $f:\kappa \to \lambda$ a map. Then there is a set $X \subseteq \kappa$ of size $\kappa$ such that $f$ is constant on $X$.  
\end{fact}

Recall the standard Erd{\"o}s partition arrow notation. Let $\kappa, \lambda, \mu$ be cardinals and $m \in \omega$. We write $\kappa \rightarrow (\lambda)^m_\mu$ if for all colourings $\chi:[\kappa]^m\to \mu$ of the set of subsets of $\kappa$ of size $m$ with $\mu$ many colours there is a set $X \subseteq \kappa$ of order type (equivalently, of size) $\lambda$ such that $\chi$ is constant on $[X]^m$. We recall also the canonical partition arrow notation, defined by Erd{\"o}s and Rado. 

\begin{definition}
    Let $\mu,\kappa$ be cardinals, $m \in \omega$ and $\chi: [\kappa]^m \to \mu $ a $\mu$-colouring of all subsets of size $m$ of $\kappa$. Fix a subset $\Delta \subseteq m$. We say that a set $X\subseteq \kappa$ is \emph{$\Delta$-canonically coloured} whenever the following holds for all increasingly enumerated $Y_1=\{a_i:i\in m\}, Y_2=\{b_i:i \in m\} \in [X]^m$:
    \begin{center}
        $\chi(Y_1)=\chi(Y_2)$ if and only if $\Delta = \{i \in m : a_i = b_i\}.$
    \end{center}
    If there exists a $\Delta\subseteq m$ such that $X$ is $\Delta$-canonically coloured, then we simply say that $X$ is \emph{canonically coloured}. For $\lambda<\kappa$ we write $\kappa \rightarrow \star (\lambda)^m$ if for all $\mu$ and all colourings $\chi: [\kappa]^m \to \mu $ there exists a set $X \subseteq \kappa$ of order type $\lambda$ such that $X$ is canonically coloured by $\chi$.
\end{definition}

Consider the special case where $m=2$. A subset $X\subseteq \kappa$ is canonical with respect to a colouring $\chi:[\kappa]^2\to \mu$ if one of the following occurs for all $i < j$ and $k < \ell$ from $X$:
\begin{enumerate}
    \item $\chi(i,j)=\chi(k,\ell)$;
    \item $\chi(i,j)=\chi(k,\ell)$ if, and only if, $i=k$;
    \item $\chi(i,j)=\chi(k,\ell)$ if, and only if, $j=\ell$;
    \item $\chi(i,j)=\chi(k,\ell)$ if, and only if, $i=k$ and $j = \ell$.
\end{enumerate}

Henceforth, we shall say that an edge colouring of a complete graph is \emph{canonical of type} $1$ (resp.\ $2,3,4$) if it satisfies condition 1 (resp.\ $2,3,4$) above for all edges, that is, all pairs $i<j$ from $\kappa$. More generally, we say that such a colouring is \emph{canonical} whenever it is canonical of any type. 

A well known result of Erd{\"o}s and Rado implies that we may always find monochromatic subsets in a sufficiently large structure. A less known result of the same authors establishes that this is also the case for canonically coloured subsets. 

\begin{theorem}[Erd{\"o}s, Rado, \cite{erdosrado}]\label{erdos}
    The following hold for $\kappa$ an infinite cardinal and $m \in \omega$:
    \begin{enumerate}[i.]
        \item $\beth_m(\kappa)^+ \rightarrow (\kappa^+)^{m+1}_{\kappa}$;
        \item $\beth_{2m+1}(\kappa)^+\rightarrow \star (\kappa^+)^{m+1}$.
    \end{enumerate}
   
\end{theorem}

The above, together with the classical Finite Ramsey Theorem and the Erd{\"o}s-Rado Canonical Ramsey Theorem \cite{erdos}, imply the following corollary which we will be making heavy use of. Recall that a cardinal $\kappa$ is said to be a \emph{strong limit cardinal} if for all cardinals $\lambda<\kappa$, $2^\lambda < \kappa$.  

\begin{fact}\label{cor:erdos}
    Let $\kappa$ be a strong limit cardinal. Then for all $m \in \omega$ and all cardinals $\lambda,\mu < \kappa$, there are cardinals $\rho_1,\rho_2<\kappa$ such that 
    \[ \rho_1 \rightarrow (\lambda)^m_{\mu}, \quad \rho_2 \rightarrow *(\lambda)^m.\]
    We write $R(m,\lambda,\mu)$ for the least such $\rho_1$ and $C(m,\lambda)$ for the least such $\rho_2$. Moreover, let $\Kcal(\lambda):=C(2,\lambda)$.
\end{fact}

We will also be making use of the following Ramsey-type lemma that is relevant for the proofs of Theorems \ref{th:smwdense} and \ref{th:main}.

\begin{lemma}\label{lem:compatible}
    Let $n \in \omega$, $\lambda$ be an infinite cardinal, and $S$ a set. Suppose that there are functions $f_{i,j}:n \to S$ for each $i<j \in \lambda$, such that $f_{i,j}(x)\neq f_{k,\ell}(x)$ for all $x \in n$ and all pairs $(i,j)\neq(k,\ell)$ with $i<j$ and $k<\ell$ from $\lambda$. Then there is a subset $X \subseteq \lambda$ of order type $\lambda$ such that $f_{i,j}(x)\neq f_{k,\ell}(y)$
    for all $x,y \in n$ and pairs $(i,j)\neq(k,\ell)$ with $i<j$ and $k<\ell$ from $X$. 
  \end{lemma}

\begin{proof}
We will say that two pairs $i<j$ and $k<\ell$ with $(i,j)\neq(k,\ell)$ from $\lambda$ are \emph{compatible} if $f_{i,j}(x)\neq f_{k,\ell}(y)$ holds for all $x,y \in n$, and otherwise we say that these pairs are \emph{incompatible}. We build a chain $I_2 \subseteq I_3 \subseteq \dots \subseteq \lambda$ using transfinite recursion, such that for all $2\leq \alpha<\lambda$:
\begin{enumerate}[i.]
    \item $I_\alpha \subseteq \max\{\omega,|a|^+\}$;
    \item $I_\alpha$ has order type $\alpha$;
    \item all pairwise distinct pairs $i<j$ and $k<\ell$ from $I_\alpha$ are compatible. 
\end{enumerate}
 Start with $I_2=\{0,1\}$. Observe that for any pair $i<j$ from $\lambda$ there are at most $n^2$ pairs $k<\ell$ that are incompatible with it. Therefore, given $I_\alpha$, there are at most $n^2 \times |\alpha^{(2)}|$ many pairs $k<\ell$ that are incompatible with some $i<j$ from $I_\alpha$. Since $I_\alpha$ has order type $\alpha$ and $\max\{\omega,|\alpha|^+\}$ is regular, the set $\{t \in \max\{\omega,|\alpha|^+\}: t > \sup I_\alpha\}$ has cardinality $\max\{\omega,|\alpha|^+\}$. It follows that there is at least one $t \in \max\{\omega,|\alpha|^+\}$ such that $t>\sup I_\alpha$, and $i<j$ and $k<t$ are compatible for all $i,j,k \in I_\alpha$. Letting $t_\alpha$ be the least such $t$, we may define $I_{\alpha+1}=I_\alpha \cup \{t_\alpha\}$. Clearly, this satisfies the conditions above.

 Suppose that $\gamma<\lambda$ is a limit ordinal and for all $\alpha<\gamma$ all $I_\alpha$ satisfy the required conditions. Let $I_\gamma = \bigcup I_\alpha$. Clearly $I_\gamma$ has order type $\gamma$, while $I_\gamma \subseteq \sup_{\alpha<\gamma} |a|^+ \subseteq |\gamma|^+$. Furthermore, all $i<j$ and $k<\ell$ are compatible for all $i<j$ and $k<\ell$ in $I_\gamma$. By the same argument, $X=\bigcup_{\alpha<\lambda} I_\alpha \subseteq \sup_{\alpha<\lambda} |a|^+=\lambda$ has order type $\lambda$ and contains only pairwise compatible pairs, satisfying the conditions of the claim. 
\end{proof}

\subsection{Algebraically universal categories}

Let $\Delta$ be a similarity type, i.e.\ a first-order language containing only function symbols. By a \emph{universal algebra of type $\Delta$}, we simply mean an $\Delta$-structure in the standard model-theoretic sense. A homomorphism of universal algebras is a map $\chi: \A \to \mathbb{B}$ such that for all tuples $\bar a$ from $\A$ and all $f \in \Delta$ it holds that 
 \[ \chi(f^A(\bar a)) = f^B(\bar b). \]
 We write $\Alg(\Delta)$ for the category of all universal algebras of similarity type $\Delta$ with homomorphisms. 

\begin{definition}
    A category $\Cfrak$ is said to be algebraically universal if for all similarity types $\Delta$ there is a fully faithful functor $\Phi:\Alg(\Delta) \to \Cfrak$.
\end{definition}

Examples of algebraically universal categories of relational structures include $\StrL$ for any $\L$ that has at least one binary relation symbol, the category of directed graphs $\overrightarrow{\Gra}$, the category of connected graphs, the category of $n$-partite undirected graphs for $n\geq 3$ (all categories considered with homomorphisms) and many others (see \cite{pultr}). On the other hand, the category of bipartite graphs is clearly not algebraically universal as any bipartite graph of size $\geq 2$ has a proper endomorphism. Likewise, the category of all total orders with order-preserving maps is not algebraically universal. Indeed, by a theorem of Dushnik and Miller \cite{linorders} it follows that every infinite linear order has a non-trivial endomorphism monoid. Hence, the only rigid total orders are the finite ones, and so for any two rigid total orders $L_1,L_2$ either there is a homomorphism $L_1 \to L_2$ or a homomorphism $L_2 \to L_1$; this condition is clearly violated by rigid graphs.

With regards to categories of algebras, $\Alg(\Delta)$ is algebraically universal if, and only if, $\sum \Delta \geq 2$, i.e.\ the arities of the function symbols add up to $\geq 2$ (Chapter II, Theorem 5.5 in \cite{pultr}). Moreover, the category of semigroups, the category of integral domains of characteristic $0$, the category of distributive lattices, and the category of boolean algebras  (all considered with homomorphisms) are algebraically universal (see \cite{pultr}). In \cite{barto}, it is shown that the category of abstract clones with clone homomorphisms is algebraically universal. On the other hand, the category of groups with group homomorphisms is clearly not algebraically universal, as every group of size $\geq 2$ has a non-trivial automorphism. 

The category $\Pos$ of posets is of particular interest. Indeed, this category is algebraically universal, but its universality depends on the existence of infinite posets. It is easy to see that any finite poset is either rigid or it has a proper endomorphism, i.e.\ an endomorphism which is not an automorphism (see e.g.\ Chapter IV, Proposition 5.8 in \cite{pultr}). It follows that no non-trivial group is the endomorphism monoid of a finite poset. A similar issue occurs with the category of semigroups. To account for this, we propose the following definition as a sensible relativisation of the notion of algebraic universality that takes into consideration the sizes of the structures in the category.  

\begin{definition}
    Let $\kappa$ be an infinite cardinal. A category $\Cfrak$ is said to be $\kappa$-algebraically universal, or $\kappa$-universal, if for all similarity types $\Delta$ there is a fully faithful functor $\Phi:\Alg(\Delta)_{<\kappa} \to \Cfrak_{<\kappa}$ from the category of universal algebras of type $\Delta$ and size $<\kappa$ into the full subcategory of $\Cfrak$ of objects of size $<\kappa$. 
\end{definition}

 In this context, the category of posets is not $\omega$-universal but it is $\kappa$-universal for all $\kappa>\omega$. Note, therefore, that being algebraically universal is not equivalent to being $\kappa$-universal for all infinite cardinals $\kappa$. Still, the category of graphs is $\kappa$-universal for all infinite $\kappa$. In practice, this implies that $\kappa$-universality is determined by whether a category fully embeds the category of graphs of size $<\kappa$.

\begin{fact}[\cite{pultr}]
A category $\Cfrak$ is $\kappa$-algebraically universal if, and only if, there is a fully faithful functor $\Phi:\Gra_{<\kappa} \to \Cfrak_{<\kappa}$.
\end{fact}

 In light of the above, we may phrase the characterisation of algebraically universal categories of graphs of Nešetřil and Ossona de Mendez as follows. 

\begin{theorem}[\cite{nesetrilossona}]
    Let $\Cfrak$ be an $\omega$-universal category of graphs. Then $\Gaif(\Cfrak)$ is $\omega$-somewhere dense. Moreover, if $\Cfrak$ is a monotone category of graphs such that $\Gaif(\Cfrak)$ is $\omega$-somewhere dense, then there is a full orientation $\tilde \Cfrak$ of $\Cfrak$ which is $\omega$-universal.
\end{theorem}

\section{Interpretable Categories}\label{sec:interpret}

In model theory, interpretations allow to definably construct a structure within a different one, possibly in a different language. This has the implication that (some of) the model-theoretic properties of the interpreted structure are carried through to the one that interprets it. In this context, we show that full embeddings give rise to primitive positive interpretations, providing a means to translate the algebraic wildness of a category into the model-theoretic wildness of its respective class of objects. This idea is present in \cite{nesetrilossona}, but we generalise it to a broader context here.  

Fix an infinite cardinal $\kappa$ and a relational language $\L'$. By an $\L_{\kappa,\kappa}$ \emph{interpretation scheme} of  $\L$-structures in $\L'$-structures we mean a tuple $\Ical=(\phi(\bar x),\psi_R(\bar y_1,\dots,\bar y_{\ar(R)}))_{R \in \L}$ of ${\L'}_{\kappa,\kappa}$-formulas with $|\bar x|=|\bar y_i|=\lambda<\kappa$. This defines a map $I:\Str(\L') \to \StrL$ such that for an $\L'$-structure $M$, $I(M)$ is the $\L$-structure on the domain $M^\lambda$ satisfying $(f_1,\dots,f_n) \in R^{I(M)}$ if, and only if, $M \models \psi_R(f_1,\dots,f_n)$ for all $R \in \L$ with $\ar(R)=n$. We say that a class $\Cfrak$ of $\L$-structures is $\L_{\kappa,\kappa}$-\emph{interpretable} in a class $\Dfrak$ of $\L'$-structures if there is an $\L_{\kappa,\kappa}$ interpretation scheme $\Ical$ such that for all $N \in \Cfrak$ there is some $M \in \Dfrak$ satisfying $N \iso I(M)$. If $\Cfrak$ is $\L_{\omega,\omega}$-interpretable in $\Dfrak$, we simply say that $\Cfrak$ is \emph{interpretable} in $\Dfrak$, or that $\Dfrak$ \emph{interprets} $\Cfrak$. 

\begin{fact}
    If $\Cfrak$ is interpretable in $\Dfrak$ and $\Dfrak$ is stable then so is $\Cfrak$.
\end{fact}

To obtain an interpretation from a full embedding, we ought to make sure that the embedded category satisfies certain criteria. A sufficient set of such criteria is provided by the following definition. 

\begin{definition}\label{def:interpretable}
    Let $\Cfrak$ be a subcategory of $\StrL$. We say that $\Cfrak$ is \emph{interpretable} if there is $\bullet \in \obj(\Cfrak)$, and for each $R \in \L$ of arity $n$ some $\Acal_R \in \obj(\Cfrak)$ such that:
    \begin{enumerate}[i.]
        \item\label{int:1} there are  homomorphisms $h_i:\bullet \to \Acal_R \in \morph(\Cfrak)$ for each $i \in n$;
        \item for all $M \in \obj(\Cfrak)$, $\Hom(\bullet,M)$ is a set and the $\L$-structure
        \[M_\bullet:=\Hom(\bullet,M) \]    
        \[(f_0,\dots,f_{n-1}) \in R^{M_\bullet} \iff \exists g:\Acal_R \to M \in \morph(\Cfrak) \text{ with } f_i=g\circ h_i\]

        is isomorphic to $M$ (in $\StrL$). 
    \end{enumerate}
\end{definition}
Intuitively, the $\bullet$ object represents a ``free'' singleton point, while each $\Acal_R$ represents a ``free'' $R$ relation in $\Cfrak$. For instance, in the category $\Gra$ of all graphs $\Acal_E$ is a directed edge, while in the category $\Sym\Gra$ of undirected graphs $\Acal_E$ is a symmetric edge (so it consists of two distinct relations). Clearly, $\StrL$ is interpretable, while other examples include the simplicial category, the category of all bipartite graphs, the category of all posets, and the category of all uniform $n$-hypergraphs amongst others. On the other hand, the category of connected undirected graphs is not interpretable because there is no $\bullet$ object. 

While the above definition could certainly be made stricter, e.g.\ by requiring that $\End(\bullet)$ is trivial, nonetheless it suffices for our purposes here, meaning that any full embedding from an interpretable category gives rise to an primitive positive $\L_{\kappa,\kappa}$ interpretation for a sufficiently large cardinal $\kappa$.

\begin{theorem}\label{th:interpret}
    Let $\Cfrak$ be an interpretable subcategory of $\StrL$, and $\Dfrak$ a subcategory of $\Str(\L')$ where $\L'$ is a relational language. Suppose that there is a full embedding $\Psi:\Cfrak \to \Dfrak$ and let $\kappa=\sup\{|\L'|^+,|\Psi(\bullet)|^+,|\Psi(\Acal_R)|^+:R \in \L\}$. Then there is a primitive positive $\L_{\kappa,\kappa}$-interpretation of $\Cfrak$ in $\Dfrak$.
\end{theorem}

\begin{proof}
    Without loss of generality, we may assume that the underlying domain of $\Psi(\bullet)$ is a cardinal $\mu<\kappa$, and likewise, that the domain of $\Psi(\Acal_R)$ is a cardinal $\lambda_R < \kappa$. Let $\bar x$ be indexed by $\mu$ and consider the formula
    
    \[
    \phi_\bullet(\bar x):=\bigwedge_{P \in \L'}\bigwedge_{(c_0,\dots,c_{\text{ar}(P)-1}) \in \mu^{ar(P)}} \bigwedge_{\Psi(\bullet) \models P(\bar c)}P(x_{c_0},\dots,x_{c_{\text{ar}(P)-1}}),\]
    
    defined such that for all $N \in obj(\Dfrak)$ and $f \in N^{\mu}$:
    
     \[ N \models \phi_\bullet(f)\iff f \in\Hom(\Psi(\bullet),N)\]
    
    Similarly, for all $R \in \L$ we define formulas $\phi_R(\bar z)$ with $|\bar z|=\lambda_R$ such that for all $f \in N^{\lambda_R}$: 
   \[ N \models \phi_R(f)\iff f\in\Hom(\Psi(\Acal_R),N)\]
    
    For all $R \in \L$ with $\ar(R)=n$, let $h_i\in \Hom(\bullet,\Acal_R)$ be the homomorphisms from (\ref{int:1}) in \Cref{def:interpretable}, and consider $\hat h_i=\Psi(h_i)\in \Hom(\Psi(\bullet),\Psi(\Acal_R))$. Fix tuples of variables $\bar y_0,\dots,\bar y_{n-1}$ labelled by $\mu$, and define formulas:
    
    \[\psi_R(\bar y_0,\dots,\bar y_{n-1}):=(\exists \bar z)(\bigwedge_{i\in n}\bigwedge_{\alpha\in \mu}[\bar y_i(\alpha)=\bar z(\hat h_i(\alpha))]\land \phi_R(\bar z)).
    \]
    
    We claim that the tuple $\Ical=(\phi_\bullet,\psi_R)_{R \in \L}$ is an $\L_{\kappa,\kappa}$ interpretation. Indeed, given $M \in \C$ consider $I\circ \Psi(M)$. This is an $\L$-structure such that for all $R\in \L$ with $\ar(R)=n$ and $f_0,\dots,f_{n-1} \in \Psi(M)^\mu$

    \[(f_0,\dots,f_{n-1})\in R^{I \circ \Psi(M)} \iff \Psi(M)\models \bigwedge_{i\in n}\phi_\bullet(f_i)\land \psi_R(f_0,\dots,f_{n-1}). \]

    The latter holds if, and only if, $f_i \in \Hom(\Psi(\bullet),\Psi(M))$ and there is a homomorphism $g \in \Hom(\Psi(\Acal_R),\Psi(M))$ such that $f_i = g \circ \hat h_i$ for all $i \in n$. Since $\Psi$ is fully faithful, this occurs if and only if there are $F_i=\Psi^{-1}(f_i) \in \Hom(\bullet,M)$ and $G=\Psi^{-1}(g) \in \Hom(\Acal_R,M)$ with $F_i=G\circ h$.

    It therefore follows that the map
    \begin{align*}
        \chi: I\circ \Psi(M) &\to M_\bullet \\
         f &\mapsto \Psi^{-1}(f)
    \end{align*}
    
    is an isomorphism of $\L$-structures, and so $I \circ \Psi(M) \iso M_\bullet \iso M$.
 \end{proof}

Consequently, if $\Psi$ maps into $\Dfrak_{<\omega}$ we obtain an interpretation of $\Cfrak$ in $\Dfrak$. In particular, this has the following implication. 

\begin{corollary}\label{cor:omegauni}
    Let $\Cfrak$ be an $\omega$-universal subcategory of $\StrL$. Then $\Cfrak$ interprets the class of all finite graphs. 
\end{corollary}

Since the class of all finite graphs is unstable and stability is preserved by interpretations, \Cref{cor:omegauni} implies that all $\omega$-universal classes of relational structures are unstable. In fact more holds. If $\Cfrak$ is a subcategory of $\StrL$ such that there is a full embedding $\Pos_{<\omega} \to \Cfrak_{<\omega}$, then $\Cfrak$ has the \emph{strict universal-order property} as defined in \cite{rosario}. The assumption here is strictly weaker as the category of posets is not $\omega$-universal. Moreover, the implication is stronger as the strict universal-order property implies several model-theoretic wildness properties, including instability (see \cite{rosario}). 

Finally, together with \Cref{monotoneNIP}, the above have the following implication.

\begin{corollary}\label{cor:omuni}
    Let $\Cfrak$ be an $\omega$-universal subcategory of $\StrL$. Then $\Gaif(\Cfrak)$ is $\omega$-somewhere dense. 
\end{corollary}

The aim of the next section is to extend \Cref{cor:omuni} to $\kappa$-universal categories, for $\kappa$ some larger infinite cardinal. The argument for the $\omega$ case relies on \Cref{monotoneNIP}, which in turn makes use of various Ramsey-type properties of $\omega$. One would expect that the same argument could be salvaged if $\kappa$ has similar properties. However, this is not quite the case; that argument would require that $\kappa \rightarrow (\kappa)^\lambda_2$ for every $\lambda<\kappa$, and while this holds for $\omega$, it is inconsistent with the Axiom of Choice for larger infinite cardinals (see \cite{kanamori}, Proposition 7.1). A different argument is therefore employed. 

\section{Somewhere density in universal categories}\label{sec:swd}

We begin by noting that the satisfaction of primitive positive formulas depends on the independent satisfaction of each of their connected components. This simple lemma is instrumental in finding paths whenever a primitive positive formula performs any kind of combinatorial coding. 

\begin{lemma}\label{ppcomponents}
Let $\kappa$ be an infinite cardinal, $\phi(\bar x)$ a primitive positive $\L_{\kappa,\kappa}$ formula, and $\Mcal_\phi$ its canonical structure. Let $(G_\gamma)_{\gamma \in \lambda}$ be the connected components of $\Gaif(\Mcal_\phi)\setminus \bar x$, and $M_\gamma = \Mcal_\phi[G_\gamma,\bar x]$. Then for any $\L$-structure $A$ and tuple $\bar a \in A$
\begin{center}
    $A \models \phi(\bar a) \iff \forall \gamma \in \lambda \ \exists \text{ homomorphism } h_\gamma:(M_i,\bar x)\to(A,\bar a).$
\end{center}
\end{lemma}

\begin{proof}
Given a collection of homomorphisms $h_\gamma:(M_\gamma,\bar x)\to (A,\bar a)$ as above, their union $h = \bigcup_{\gamma\in\lambda} h_\gamma$ is a homomorphism $(\Mcal_\phi,\bar x)\to (A,\bar a)$. This is well-defined as the components $G_\gamma$ are disjoint, and hence $M_\gamma\cap M_\delta=\bar x$ for $\gamma\neq \delta$, while $\bar x$ has the same image under all the $h_\gamma$.  Also, if $(v_1, \dots, v_n) \in R^{\Mcal_\phi}$ for some $R \in \L$, then the elements in $\bar v \setminus \bar x$ lie in the same connected component of $\Gaif(\Mcal_\phi)\setminus \bar x$, say $G_\alpha$. Since $h_\alpha$ is a homomorphism, it follows that $h_\alpha(\bar v)=h(\bar v) \in R^A$. 
\end{proof}

With this, we argue that whenever there is a full embedding $\Card \to \Cfrak$ from the category of cardinals viewed as $\in$-structures with order-preserving maps, then $\Gaif(\Cfrak)$ is totally somewhere dense. Observe that the assumption is strictly weaker than algebraic universality, as the category of cardinals is not algebraically universal. On the other hand its monotone closure is, as argued in \Cref{sec:wellfdd}. This gap reflects the fact that nowhere density is essentially a property of the monotone closure of a class. 

First, we show this relativised to categories of structures of size bounded by an inaccessible cardinal. Recall that a cardinal is said to be inaccessible if it is uncountable, a strong limit, and regular, i.e.\ equal to its cofinality. The existence of inaccessibles is independent of $\ZFC$ (see \cite{kanamori}). Since algebraic universality is not equivalent to $\kappa$-universality for all cardinals $\kappa$, we choose a relativised proof to exhibit exactly those levels where our theorem holds. Nonetheless, it will become apparent that this also holds without the relativisation to inaccessibles. For clarity, the proof is divided into a sequence of claims. 

\begin{theorem}\label{th:smwdense}
    Let $\kappa$ be an inaccessible cardinal and $\Cfrak$ a subcategory of $\StrL$ such that $|\L|<\kappa$ and there is a full embedding $\Psi:\Card_{<\kappa} \to \Cfrak_{<\kappa}$. Then $\Gaif(\Cfrak)$ is $\kappa$-somewhere dense. 
\end{theorem}

\begin{proof}
    Since $\Card_{<\kappa}$ is interpretable, it follows by the assumption and \Cref{th:interpret} that there is some primitive positive $\L_{\kappa,\kappa}$ interpretation of $\Card_{<\kappa}$ in $\Cfrak_{<\kappa}$. So, there is a formula $\phi(\bar x,\bar y)=\exists \bar w \psi(\bar x,\bar y,\bar w)$ where $\bar x$ and $\bar y$ are enumerated by $\mu <\kappa$, and for every cardinal $\lambda < \kappa$  some $M_\lambda \in \obj(\Cfrak)$ and functions $(f_i)_{i \in \lambda}$ from $(M_\lambda)^\mu$ such that
    \[  M_\lambda \models \phi(f_i,f_j) \iff i< j \quad (\star)\]

    For clarity, we write $f_i^\lambda$ to highlight that these come from $(M_\lambda)^\mu$. Enumerate the variables in $\bar w$ by $\nu < \kappa$, and pick existential witnesses $h^\lambda_{i,j}$ from $(M_\lambda)^\nu$ such that 
    \[ i < j \implies M_\lambda \models \psi(f^\lambda_i,f^\lambda_j,h^\lambda_{i,j}). \]

    For $i<j<k<\ell\in\lambda$, let $\Delta_\lambda(i,j,k,\ell)$ be the equality type of the tuple $(f^\lambda_i,f^\lambda_j,f^\lambda_k,f^\lambda_\ell,h^\lambda_{i,j},h^\lambda_{k,\ell},h^\lambda_{j,k},h^\lambda_{j,\ell})$, that is, the set of atomic formulas $\eta(\bar x_1,\bar x_2,\bar x_3,\bar x_4,\bar w_1,\bar w_2,\bar w_3,\bar w_4)$ using only the equality symbol such that $M_\lambda \models \eta(f^\lambda_i,f^\lambda_j,f^\lambda_k,f^\lambda_\ell,h^\lambda_{i,j},h^\lambda_{k,\ell},h^\lambda_{j,k},h^\lambda_{j,\ell})$. 

\begin{claim}\label{th1:claim1}
    We may assume that $\Delta_\lambda(i,j,k,\ell)$ is constant for all $\lambda<\kappa$ and $i<j<k<\ell \in \lambda$. 
\end{claim}

\begin{claimproof}
    Let $\delta=2^{\max\{\mu,\nu\}}$ and $r(\lambda)=R(4,\lambda,\delta)$. Observe that there are at most $\delta$ many equality types, while $\delta$ and $f(\lambda)$ are strictly smaller than $\kappa$ since $\kappa$ is a strong limit. Therefore, starting from $M_{r(\lambda)}$ we may colour quadruples $i<j<k<\ell \in r(\lambda)$ according to $\Delta_\lambda$. By \Cref{cor:erdos}, we may find set $I\subseteq r(\lambda)$ of order type $\lambda$ which is monochromatic in this colouring. We may therefore pass to the subsequence $(M_{r(\lambda)})_{\lambda <\kappa}$ and relabel, so that without loss of generality $\Delta_\lambda(i,j,k,\ell)$ is constant for all $i<j<k<\ell \in \lambda$. Furthermore, since the number of equality types is at most $\delta<\kappa$ and $\kappa$ is regular, it follows by the pigeonhole principle that there is a set $X \subseteq \kappa$ of size $\kappa$ such that $\Delta_{\lambda}$ is constant for all $\lambda \in X$. After passing to the subsequence $(M_{\lambda})_{\lambda \in X}$ and relabelling, we may finally assume that $\Delta_\lambda(i,j,k,\ell)$ is constant for all $i<j<k<\ell \in \lambda$ and all $\lambda<\kappa$. 
\end{claimproof}

Henceforth, we unambiguously write $\Delta$ for the above. Let $S=\{s \in \mu : (\bar x_1(s) = \bar x_2(s)) \in \Delta\}$, and $T = \{t \in \nu : (\bar w_1(t) = \bar w_2(t)) \in \Delta\}$. Observe that by definition these satisfy
 \[f_i^\lambda(s) = f^\lambda_j(s) \iff s \in S;\]
 \[h^\lambda_{i,j}(t) = h^\lambda_{k,\ell}(t) \iff t \in T,\]
 for all $i<j<k<\ell \in \lambda$.

\begin{claim}\label{th1:claim2}
    There is $m \in \omega$ such that for all $\lambda \in \kappa$ and $i<j \in \lambda$ there are injective maps $p^\lambda_{i,j}:m+2 \to M_\lambda$ satisfying the following for all $i<j$ and $k<\ell$ from $\lambda$:
    \begin{enumerate}[i.]
         \item\label{cond:iv} $p^{\lambda}_{i,j}(0),p^{\lambda}_{i,j}(1),\dots,p^{\lambda}_{i,j}(m+1)$ is a path in $\Gaif(M_\lambda)$;
        \item\label{cond:i} $p^{\lambda}_{i,j}(0)=p^{\lambda}_{k,\ell}(0)\iff i=k$;
        \item\label{cond:ii} $p^{\lambda}_{i,j}(m+1)=p^{\lambda}_{k,\ell}(m+1)\iff j=\ell$;
        \item\label{cond:iii} $p^{\lambda}_{i,j}(n)\neq p^{\lambda}_{k,\ell}(n)$ for $i<j<k<\ell$ and $n \in [m]$;
        \item\label{cond:v} the equality type of ${p^{\lambda}_{i,j}}^\frown {p^{\lambda}_{k,\ell}}^\frown {p^{\lambda}_{j,k}}^\frown p^{\lambda}_{j,\ell}$ is constant for $i<j<k<\ell<\lambda$.
  \end{enumerate}
\end{claim}

\begin{claimproof}
    We first argue that there are $\alpha,\beta \in \mu\setminus S$ and a path in $\Gaif(\Mcal_\phi)$ from $\bar x(\alpha)$ to $\bar y(\beta)$ passing only through the elements in $\bar w$, and avoiding the variables $\bar x(s)$, $\bar y(s)$ for $s \in S$, and $\bar w(t)$ for $t \in T$. Indeed, consider $G = \Gaif(\Mcal_\phi)\setminus (\{ \bar x(s),\bar y(s) : s \in S \}\cup \{\bar w(t) :t \in T\})$, $N_x = N^G(\{\bar x(s):s\notin S\})$, $N_y=N^G(\{\bar y(s): s\notin S\})$. Assume for a contradiction that $N_x$ and $N_y$ are disjoint. Since $M_\lambda \models \psi(f^\lambda_1,f^\lambda_2,h^\lambda_{1,2})\land \psi(f^\lambda_3,f^\lambda_4, h^\lambda_{3,4}$), it follows that $M \models \phi(f^\lambda_3, f^\lambda_2)$ by Lemma \ref{ppcomponents}. This contradicts $(\star)$. Consequently, $N_x\cap N_y \neq \emptyset$, and so there are $\alpha,\beta \in \mu\setminus S$ such that there is a path in $G$ from $\bar x(\alpha)$ to $\bar y(\beta)$, as required. 

    Let $P \subseteq \nu \setminus T$ be the finite set of indices of the variables in $\bar w$ that appear on this path, and write $m=|P|$. Without loss of generality, we may assume that the path above starts at $\bar x(\alpha)$ and finishes at $\bar y(\beta)$, passing only through variables $\bar w(p)$ for $p \in P$. Indeed, if the current path does not satisfy this, there are $u'$ and $v'$ in $\mu$ and a sub-path starting from $\bar x(u')$ and finishing at $\bar y(v')$ passing only through variables in $\bar w$. We may then consider this sub-path and relabel appropriately. 

    Let $P=\{p_1,\dots,p_m\}$ be the enumeration of $P$ in the order of the corresponding path $\bar x(\alpha),\bar w(p_1),\dots,\bar w(p_m),\bar y(\beta)$ in $\Gaif(\Mcal_\phi)$. It follows that 
    $f_i^\lambda(\alpha),h^\lambda_{i,j}(p_1),\dots,h^\lambda_{i,j}(p_m),f^\lambda_j(\beta)$ is a path in $\Gaif(M_\lambda)$ for all $\lambda<\kappa$ and $i<j <\lambda$. Notice that $f_i^\lambda(\alpha)\neq f_j^\lambda(\beta)$; indeed if that was the case then for any $j'>j$ it would hold that $f_j^\lambda(\beta)=f_{j'}^\lambda(\beta)$ by \Cref{th1:claim1}, contradicting that $\beta \notin S$. So, without loss of generality we may assume that the above path contains no repeated vertices; otherwise we may consider a proper sub-path from $f_i^\lambda(\alpha)$ to $f^\lambda_j(\beta)$, and since all such paths have the same equality type by \ref{th1:claim1}, exactly the same sub-path may be considered for all $i<j$ and $\lambda <\kappa$. 
    Hence, for all $\lambda<\kappa$ and pairs $i<j$ from $\lambda$ define the injective maps 
    \begin{align*}
        p^\lambda_{i,j}:m+2 &\to M_\lambda   \\
        0 &\mapsto f_i^{\lambda}(\alpha)   \\
        n \in [m] &\mapsto h_{i,j}^\lambda(p_n)   \\
        m+1 &\mapsto f_i^{\lambda}(\beta).
    \end{align*}
    Condition (\ref{cond:iv}) trivially holds, and moreover, since $\alpha,\beta \notin S$, conditions (\ref{cond:i}) and (\ref{cond:ii}) hold. Also, condition (\ref{cond:iii}) holds since $P \subseteq \nu \setminus T$,
    while condition (\ref{cond:v}) holds by \Cref{th1:claim1}.  
\end{claimproof}

We proceed to show that we may strengthen (\ref{cond:iii}). 

\begin{claim}\label{thm1:claim3}
    We may assume that $p^\lambda_{i,j}(n_1) \neq p^\lambda_{k,\ell}(n_2)$, for all $n_1,n_2 \in [m]$ and any pairs $(i,j)\neq (k,\ell)$ with $i<j$ and $k<\ell$ from $\lambda$. 
\end{claim}

\begin{claimproof}
   Define colourings $\chi_{\lambda,n}(i,j) = p^\lambda_{i,j}(n)$ for each $n \in m+2$ and $\lambda<\kappa$. Having started with a suitably large $\rho(\lambda)<\kappa$, we may iterate the canonical Erd{\"o}s-Rado theorem from \Cref{cor:erdos} $(m+2)$ times using these colourings. Indeed, $\rho(\lambda)=\beth_\omega(\lambda)$ works, since $\cf(\kappa)>\omega$ implies that $\rho(\lambda)<\kappa$. Relabelling once more, we therefore have a map $ t_\lambda: [m+2] \to [4]$ for every $\lambda<\kappa$ such that whether $p^\lambda_{i,j}(n)=h^\lambda_{k,\ell}(n)$ depends on the canonical case $t_\lambda(n)$. Since $4^{m+2}$ is finite, it follows by the pigeonhole principle that there is a a set $X \subseteq \kappa$ of size $\kappa$ such that $t_{\lambda}$ is the same for all $\lambda \in X$. By restricting to the subsequence $(M_{\lambda})_{\lambda \in X}$ and relabelling, we may assume that $t_\lambda$ is the same for all $\lambda<\kappa$. Write $t:[m+2] \to [4]$ for this map. 
    
    Observe that conditions (\ref{cond:i}) and (\ref{cond:ii}) ensure that $t(0)=2$ and $t(m+1)=3$. Moreover, $t(n)\neq 1$ for all $n \in [m]$. If not, then for some $n \in [m]$ it holds that $p^\lambda_{i,j}(n)=p^\lambda_{k,\ell}(n)$ for all $i<j$ and $k<\ell$; however $h^\lambda_{i,j}(p) \neq h^\lambda_{k,l}(p)$ for any $i<j<k<\ell$ by condition (\ref{cond:iii}). Clearly, the claim follows if $t(n)=4$ for all $n \in [m]$. 

    So, suppose that there is at least one $n \in [m]$ that falls into the canonical case $(2)$, i.e.\ $t(n)=2$, and let $c \in [m]$ be the greatest such that $n$. It follows that $p^\lambda_{i,j}(c)=h^\lambda_{k,\ell}(c)$ if, and only if, $i=k$. Since $t(m+1)=3$, there exists a least $d>c$ such that $t(d)=3$. Then, define $\tilde p^\lambda_{i,j}:(d-c+1)\to M_\lambda$ by $n \mapsto p^\lambda_{i,j}(n+c)$; since $t(d)=3$, $p^\lambda_{i,j}(d)=p_{k,\ell}^\lambda(d)$ if and only if $j=\ell$, and so the maps $\tilde p^\lambda_{i,j}$ satisfy the conditions of \Cref{th1:claim2}. Since $t(n)=4$ for all $n \in [d-c-1]$, the claim follows in this case by considering the maps $\tilde p^\lambda_{i,j}$. Finally, the case where there is at least one $n \in [m]$ that falls into Case $(3)$ is handled symmetrically.

    We have thus ensured that $p^\lambda_{i,j}(n) \neq p^\lambda_{k,\ell}(n)$, for all $n \in [m]$ and any pairs $(i,j)\neq (k,\ell)$ with $i<j$ and $k<\ell$ from $\lambda$. Consequently, \Cref{lem:compatible} implies that for all $\lambda<\kappa$ there is a set $I_\lambda \subseteq \lambda$ of order type $\lambda$ such that $p^\lambda_{i,j}(n_1)\neq p^\lambda_{k,\ell}(n_2)$ for all $n_1$ and $ n_2$ from $[m]$ and all pairs $(i,j)\neq (k,\ell)$ with $i<j$ and $k<\ell$ from $I_\lambda$. The claim thus follows by restricting on these sets and relabelling appropriately. 
\end{claimproof}

\begin{claim}
    The maps $p^\lambda_{i,j}$ satisfy $p^\lambda_{i,j}(0)\neq p_{k,\ell}^\lambda(n)$ and $p^\lambda_{i,j}(m+1)\neq p_{k,\ell}^\lambda(n)$ for all $n \in [m]$ and all pairs $0<i<j$ and $0<k<\ell$ from $\lambda$. 
\end{claim}

\begin{claimproof}
    We first argue for the former. Assume for a contradiction that  $p^\lambda_{i,j}(0)=p^\lambda_{k,\ell}(n)$ for some $n \in [m]$ and $i<k<\ell$. Then $p^\lambda_{i,j}(0)=p^\lambda_{i,k}(0)=p^\lambda_{k,\ell}(n)$, and so $p^\lambda_{i,k}(0)=p^\lambda_{k,\ell+1}(n)$ by condition (\ref{cond:v}). It follows that $p^\lambda_{k,\ell}(n)=p^\lambda_{k,\ell+1}(n)$, contradicting \Cref{thm1:claim3}. Likewise, if $p^\lambda_{i,j}(0)=p^\lambda_{k,\ell}(n)$ for some $n \in[m]$ and $k<i<\ell$, then $p^\lambda_{i,j}(0)=p^\lambda_{i,\ell}(0)=p^\lambda_{k,\ell}(n)$ and so it holds that $p^\lambda_{k,\ell}(n)=p^\lambda_{i,\ell}(0)=p^\lambda_{0,\ell}(n)$ by condition (\ref{cond:v}). Again, this contradicts \Cref{thm1:claim3}. Finally, if $p^\lambda_{i,j}(0)=p^\lambda_{k,\ell}(n)$ for some $n \in [m]$ and $k<\ell<i$, then it holds that $p^\lambda_{k,\ell}(n)=p^\lambda_{i,j}(0)=p^\lambda_{0,\ell}(n)$ by condition (\ref{cond:v}), contradicting \Cref{thm1:claim3} once more. The latter claim follows analogously by analysing the cases $j<k<\ell$, $k<j<\ell$, and $k<\ell<j$. 
\end{claimproof}
 
Consider the elements $a^\lambda_i := p^\lambda_{i,j}(0) \in M_\lambda$ and $b^\lambda_i:=p^\lambda_{0,i}(m+1) \in M_\lambda$ for $0<i<j$. These are well-defined by conditions (\ref{cond:i}) and (\ref{cond:ii}). Notice that $a^\lambda_i \neq b^\lambda_j$ for any $i<j$ since each $p^\lambda_{i,j}$ is injective. Moreover, $a^\lambda_i \neq b^\lambda_j$ for $j<i$; if this was the case then $p^\lambda_{0,j}(m+1)=p^\lambda_{i,i+2}(0)$, and so condition (\ref{cond:v}) would imply that $p^\lambda_{i,i+2}(0)=p^\lambda_{x_\lambda,j}(m+1)=p^\lambda_{i+1,i+2}(0)$, contradicting condition (\ref{cond:ii}). Finally, there are two cases to distinguish depending on whether $a^\lambda_i = b^\lambda_i$ or not. 

If $a^\lambda_i=b^\lambda_i$ for some $\lambda<\kappa$ and $i \in \lambda$ then it follows by \Cref{th1:claim1} that $a^\lambda_i=b^\lambda_i$ for all $\lambda<\kappa$ and all $i\in \lambda$. Hence, for all $\lambda<\kappa$, the elements $a^\lambda_j := p_{i,j}^\lambda(m+1)$ are the native vertices of an $m$-subdivided clique of size $\lambda$ in $\Gaif(M_\lambda)$. Indeed, we have ensured that $a^\lambda_i\neq a^\lambda_j$ for all $i<j$, while the maps $p^\lambda_{i,j}$ specify pairwise disjoint paths of length $m+1$ from $a^\lambda_i$ to $a^\lambda_j$. Since this holds for all $\lambda<\kappa$, this implies that $\Gaif(\Cfrak)$ is $\kappa$-somewhere dense. On the other hand, if $a^\lambda_i\neq b^\lambda_i$ for all $\lambda<\kappa$ and all $i\in \lambda$, then the elements $a^\lambda_i$ and $b^\lambda_j$ are the native vertices of an $m$-subdivided half-graph of size $\lambda$ in $\Gaif(M_\lambda)$. Since any $m$-subdivided half-graph of size $\lambda^+$ contains 
a $2m$-subdivided clique of size $\lambda$ and $\kappa$ is a limit cardinal, it follows that $\Gaif(\Cfrak)$ is $\kappa$-somewhere dense. 
\end{proof}

\begin{corollary}
    Let $\kappa$ be an inaccessible cardinal and $\Cfrak$ a $\kappa$-universal subcategory of $\StrL$ such that $|\L|<\kappa$. Then $\Gaif(\C)$ is $\kappa$-somewhere dense. 
\end{corollary}

Note that the proof of \Cref{th:smwdense} given above relies on \Cref{pigeonhole} and \Cref{cor:erdos}. Both of these hold for unbounded $\kappa$. The unbounded version of \Cref{cor:erdos} is precisely the Erd{\"o}s-Rado theorem (\Cref{erdos}), while the unbounded version of \Cref{pigeonhole} is the following trivial statement.

\begin{fact}\label{fact:pigeon}
    Let $\lambda$ be a cardinal, and $\phi:\Card \to \lambda$ a class function. Then there is some $\mu<\lambda$ such that $\phi^{-1}[\mu]$ is unbounded, i.e.\ for all cardinals $\alpha$ there is some $\beta>\alpha$ such that $\phi(\beta)=\mu$. 
\end{fact}

Hence, the proceeding corollary follows by using the full Erd{\"o}s-Rado theorem and the above fact in place of \Cref{cor:erdos} and \Cref{pigeonhole} respectively in the proof of \Cref{th:smwdense}.

\begin{corollary}\label{cor:swd}
    Let $\Cfrak$ be an algebraically universal subcategory of $\StrL$. Then $\Gaif(\C)$ is totally somewhere dense. 
\end{corollary}

The aim of the remaining of this paper is to show a partial converse to this. Clearly, a full converse is out of reach. For instance, the category of all bipartite undirected graphs is totally somewhere dense but not algebraically universal. However, we show that we may achieve algebraic universality after assigning an ``orientation'' to the objects of any totally somewhere dense category. Towards this, \Cref{sec:gadgets} first develops the necessary machinery to establish algebraic universality for categories of relational structures.

\section{The category of gadgets}\label{sec:gadgets}

A standard technique used to show that a category is algebraically universal is the so-called \emph{arrow construction}. In its simplest form, the idea is to glue a copy of a fixed ``arrow'' (which could be a graph, algebra, topological space, etc.) to every edge of a graph, so that every homomorphism of the underlying graphs gives a morphism of the resulting structures after gluing the arrow. This construction is therefore functorial, and under additional assumptions on the arrow, it can be shown to preserve algebraic universality. 

In most contexts, different copies of the arrow do not share any points in common unless they are glued on edges with common vertices. Here, we extend the arrow construction to relational structures by permitting different copies of the arrow to share multiple vertices in a way that reflects the cases of the canonical Erd{\"o}s-Rado theorem. Although this choice seems arbitrary at first, its usefulness will become apparent in \Cref{sec:monotone}. 

\begin{definition}
    Let $M$ be an $\L$-structure, $A, B, P$ disjoint subsets of $M$, and $\alpha,\beta \in M$ two distinct elements not in $A,B,P$. We call the tuple $(M,\alpha,\beta,A,B,P)$ an \emph{$\L$-gadget}. A homomorphism between two $\L$-gadgets $(M,\alpha_1,\beta_1,A_1,B_1,P_1)$ and $(N,\alpha_2,\beta_2,A_2,B_2,P_2)$ is a map $\rho:M\to N$ which is a homomorphism of the underlying $\L$-structures, and moreover 
    \[ \rho: \alpha_1 \mapsto \alpha_2, \beta_1 \mapsto \beta_2, A_1 \mapsto A_2, B_1 \mapsto B_2, P_1 \mapsto P_2, \]

    where the latter means that $\rho[A_1]\subseteq A_2$, and likewise for $B_i$ and $P_i$. We write $\GdgL$ for the category of $\L$-gadgets with homomorphisms. Moreover, we say that the $\L$-gadget is \emph{simple} whenever $A,B,P$ are empty.
\end{definition}

    Clearly, there is a natural forgetful functor $\mathbf{F}:\GdgL \to \StrL$ that maps an $\L$-gadget to the underlying $\L$-structure. The following definition and subsequent lemma illustrate how an $\L$-gadget gives rise to a functor from the category of graphs.  
    
 \begin{definition}   
    Fix an $\L$-gadget $(M,\alpha,\beta,A,B,P)$ and a graph $(G,E)$. Let
    \[G \star M := G \sqcup (\pi_0[E]\times A)\sqcup (\pi_1[E]\times B) \sqcup (E \times (M\setminus (\{a,b\}\cup A\cup B \cup P))) \sqcup P,\]
    where $\pi_i$ is the $i$-th projection function.
    For an edge $(u,v) \in E$, define the map $\phi_{(u,v)}^{G,M}:M \to G \star M$ by:

    \[   
\phi_{(u,v)}^{G,M}(x) = 
     \begin{cases}
       u, &\quad \text{if } x=\alpha  \\
       v, &\quad \text{if } x=\beta \\
       x,&\quad \text{if } x \in P \\
       (u,x),&\quad \text{if } x \in A \\
       (v,x), &\quad \text{if } x \in B \\ 
       (u,v,x), &\quad \text{otherwise.}
     \end{cases}
\]
  
   Define an $\L$-structure on $G\star M$ such that for every $R \in \L$
   \[ \bar y \in R^{G \star M} \iff \text{ there is } \bar x \in R^M \text{ and } (u,v) \in E \text{ such that } \phi_{(u,v)}(\bar x) = \bar y.\]
   For a graph homomorphism $f:(G,E)\to (H,S)$ define the map 
   \[ f\star M : G \star M \to H \star M, \text{by}\]

    \[   
(f \star M)(x) = 
     \begin{cases}
       f(x), &\quad \text{if } x \in G  \\
       x,&\quad \text{if } x \in P \\       
       (f(u),a), &\quad \text{if } x=(u,a) \\
       (f(v),b), &\quad \text{if } x=(v,b) \\ 

       (f(u),f(v),c), &\quad \text{if } x=(u,v,c).
     \end{cases}
\]
\end{definition}

\begin{lemma}
    For every $(u,v) \in E, \phi_{(u,v)}:M \to G\star M$ is an injective strong homomorphism. For a graph homomorphism $f: (G,E)\to (H,S)$, $f\star M: G\star M \to H \star M$ is a homomorphism of $\L$-structures. 
\end{lemma}

\begin{proof}
    That $\phi_{(u,v)}$ is an injective strong homomorphism follows trivially from the definition of $G \star M$. If $f: (G,E) \to (H,S)$ is a homomorphism of graphs, then for every $(u,v) \in E$ it holds that 
    \[ (f\star M)\circ \phi_{(u,v)}^{G,M}= \phi_{(f(u),f(v))}^{M,H}.\]
    Let $R \in \L$ and $\bar y \in R^{G\star M}$. Then there is a $\bar x \in R^M$ and $(u,v) \in E$ such that $\phi_{(u,v)}(\bar x)=\bar y$. It follows that
    \[ (f \star M)(\bar y)=\phi_{(f(u),f(v))}(\bar x),  \]
    and therefore $(f \star M)(\bar y) \in R^{H \star M}$.
\end{proof}

Hence, for any fixed $\L$-gadget $M$ the assignment $(-\star M):\Gra \to \StrL$ is functorial. As it turns out, the same holds for the assignment $(G \star -): \GdgL \to \StrL$ for any fixed graph $G$. 

\begin{definition}
    Fix a graph $G$, and let
    \[ \rho: (M,\alpha_1,\beta_1,A_1,B_1,P_1)\to (N,\alpha_2,\beta_2,A_2,B_2,P_2)\]
    be a homomorphism of $\L$-gadgets. Define the map 
    \[ G\star \rho : G \star M \to G \star N, \text{by}\]

    \[   
(G \star \rho)(x) = 
     \begin{cases}
       x, &\quad \text{if } x \in G  \\
       \rho(x),&\quad \text{if } x \in P \\       
       (u,\rho(a)), &\quad \text{if } x=(u,a) \\
       (v,\rho(b)), &\quad \text{if } x=(v,b) \\ 

       (u,v,\rho(c)), &\quad \text{if } x=(u,v,c).
     \end{cases}
\]
\end{definition}

\begin{lemma}
    For $G$ and $\rho$ as above, $(G\star \rho)$ is a homomorphism of $\L$-structures. 
\end{lemma}

\begin{proof}
    Firstly, $(G\star \rho)$ is well defined since $\rho$ satisfies
    \[ \rho: \alpha_1 \mapsto \alpha_2, \beta_1 \mapsto \beta_2, A_1 \mapsto A_2, B_1 \mapsto B_2, P_1 \mapsto P_2. \]
    Moreover, for every $(u,v) \in E$ it holds that 
    \[ (G \star \rho) \circ \phi^{G,M}_{(u,v)} = \phi^{N,G}_{(u,v)}\circ \rho.\]
    Consequently, if $\bar a \in R^{G\star M}$ for some $R \in \L$, then there is some $\bar x \in R^M$ and a $(u,v) \in E$ such that $\phi_{(u,v)}(\bar x)=\bar a$. It follows that $\rho(\bar x) \in R^N$, so 
    \[\phi_{(u,v)}\circ \rho(\bar x) = (G\star \rho)\circ \phi_{(u,v)}(\bar x)=(G \star \rho)(\bar a) \in R^{G\star N}.\]
    So, $(G\star \rho)$ is a homomorphism of $\L$-structures.
\end{proof}

In fact, even more is true. We write $\Gra^+$ for the category of graphs having at least one edge with homomorphisms. 

\begin{proposition}\label{prop:faithful}
    $\star : \Gra^+ \times \GdgL \to \StrL$ is a faithful bifunctor. 
\end{proposition}

\begin{proof}
    Let $f:G\to H$ be a graph homomorphism, and $\rho:(M,\alpha_1,\beta_1,A_1,B_1,P_1)\to (N,\alpha_2,\beta_2,A_2,B_2,P_2)$ a homomorphism of $\L$-gadgets. Define the map
    
    \[f \star \rho:G\star M\to H \star N \text{ by:}\]
    
    \[   
(f \star \rho)(x) = 
     \begin{cases}
       f(x), &\quad \text{if } x \in G  \\
       \rho(x),&\quad \text{if } x \in P \\       
       (f(u),\rho(a)), &\quad \text{if } x=(u,a) \\
       (f(v),\rho(b)), &\quad \text{if } x=(v,b) \\ 

       (f(u),f(v),\rho(c)), &\quad \text{if } x=(u,v,c).
     \end{cases}
\]
    As with the simpler cases, it holds that for all $(u,v)\in E(G)$
    \[ (f \star \rho)\circ \phi^{G,M}_{(u,v)}=\phi^{N,G}_{(f(u),f(v))}\circ \rho,\]
    from which it can be easily checked that $(f \star \rho)$ is a homomorphism of $\L$-structures. 

    Let $(f,\rho)\neq (g,\mu) \in \Hom((G,M),(H,N))$. If $f\neq g$, then clearly $(f \star \rho)\neq(g\star \mu)$. Moreover, if $\rho\neq\mu$ then $M$ must necessarily contain points other than $\alpha_1$ and $\beta_1$. So there is some $x \notin \{\alpha_1,\beta_1\}$ such that $\rho(x)\neq\mu(x)$. Since $G$ has at least one edge, there is some $(u,v) \in G$ such that $x$ is equal to one of $(u,a),(v,b),(u,v,c)$ or $p$, for some $p \in P$. It therefore follows that $(f \star \rho)(x)\neq (g \star \mu)(x)$. 
\end{proof}

We are forced to restrict to graphs with at least one edge so as to ensure that $G\star M\neq G$, and hence faithfulness of $(-\star-)$. In practice, we will also often require that neither $G$ nor $M$ contain isolated points so as to avoid unwanted endomorphisms of the resulting structure. The following lemma illustrates that in this context, the $\L$-structure $G \star M$ may be defined in terms of a universal property. 

\begin{lemma}\label{lem:injective}
    Fix an $\L$-gadget $(M,\alpha,\beta,A,B,P)$, a graph $G$ without isolated points, and an $\L$-structure $N$. Suppose that for each $(u,v)\in E(G)$ there are injective homomorphisms $f_{u,v}:M \to N$ such that $f_{u,v}(x)=f_{s,t}(y)$ if, and only if, $x=y$ and one of the following holds:
    \begin{itemize}
        \item $(u,v)=(s,t)$ and $x \in M$;
        \item $u = s, v \neq t$ and $x \in \{\alpha\}\cup A\cup P$;
        \item $u \neq s, v = t$ and $x \in \{\beta\}\cup B \cup P$;
        \item $u \neq s, v \neq t$ and $x \in P$.
    \end{itemize}
    Then there is an injective homomorphism $G \star M \to N$. 
\end{lemma}

\begin{proof}
    Consider the map 
    \[ \Phi:= \bigcup_{(u,v) \in E(G)} f_{u,v}\circ (\phi^{M,G}_{u,v})^{-1}: G \star M \to N.\]
    Since $G$ contains no isolated points, it follows that $\Phi$ is defined on all of $G \star M$. We argue that it is well-defined. 

    Indeed, 
    \[\phi_{u,v}[M]\cap\phi_{s,t}[M] = 
    \begin{cases}
        \phi_{u,v}[M] &\quad \text{if }(u,v)=(s,t); \\
        \{u\} \cup \{(u,a):a \in A\}\cup P &\quad \text{if }u=s \land v \neq t; \\
        \{v\} \cup \{(v,b):b \in B\}\cup P &\quad \text{if }u\neq s \land v = t; \\
        P &\quad \text{otherwise.}
    \end{cases}
    \]
    Since the maps $f_{u,v}$ and $f_{s,t}$ agree on these four cases, it follows that $\Phi$ is well-defined. Moreover, since each $f_{u,v}\circ (\phi^{M,G}_{u,v})^{-1}$ is injective and $f_{u,v}\circ (\phi^{M,G}_{u,v})^{-1}(x)\neq f_{s,t}\circ (\phi^{M,G}_{s,t})^{-1}(y)$ whenever $x \neq y$ or $x=y$ and $x \notin \phi_{u,v}[M]\cap\phi_{s,t}[M]$ by the assumption, it follows that $\Phi$ is injective. Finally, $\Phi$ is a homomorphism each $f_{u,v}$ and $\phi_{u,v}^{M,G}$ is, and by construction all tuples that appear in a relation in $G \star M$ are contained in $\phi_{u,v}[M]$ for some $(u,v) \in E(G)$.  
\end{proof}

As previously mentioned, the usefulness of the generalised arrow construction comes from the fact that, under few additional assumptions, the construction is not only functorial, but in fact preserves universality. Here, a sufficient set of such assumptions is given by the following theorem, which will be instrumentally used in \Cref{sec:monotone} but is also of independent interest. 

\begin{theorem}\label{th:system}
    Fix some subcategory $\Gfrak$ of $\Gra$ such that no $G \in \obj(\Gfrak)$ has isolated points. Let $\Cfrak$ be a subcategory of $\StrL$, and suppose that there is some $\L$-gadget $(M,\alpha,\beta,A,B,P)$ such that
    \begin{enumerate}
        \item $M \in \obj(\Cfrak)$ and for every $G \in \obj(\Gfrak)$, $G \star M \in \obj(\Cfrak)$;
        \item\label{2} $\phi_{(u,v)}^{G,M}$ are the only morphisms $M \to G\star M$ in $\Cfrak$;
        \item\label{3} $f \star M$ is in $\morph(\Cfrak)$ for every $f \in \morph(\Gfrak)$;
    \end{enumerate}
    Then the functor $\Phi: \Gfrak \to \Cfrak$ given by 
    \[ \Phi(G) = G \star M, \quad \Phi(f)=f\star M \]
    is a full embedding. Consequently, if $\Gfrak$ is algebraically universal then so is $\Cfrak$. Moreover, if $|M|<\kappa$ for some infinite cardinal $\kappa$ and $\Gfrak$ is $\kappa$-universal, then so is $\Cfrak$.
\end{theorem}

\begin{proof}
    It suffices to show that the functor $\Phi: \Gfrak \to \Cfrak$
    is a full embedding. The rest of the claim follows trivially by definitions.
    
    By \Cref{prop:faithful}, $*:\Gra^+ \times \GdgL \to \StrL$ is a faithful bifunctor, and so $\Phi$ is faithful. It therefore remains to show fullness. Let 
    \[g : (G,E)\star M \to (H,S) \star M\]
    be a morphism in $\Cfrak$. It follows by (\ref{2}), that for every edge $(u,v) \in E$ there is an edge $(u',v') \in S$ such that $g \circ \phi_{(u,v)}=\phi_{(u',v')}$. Define the map 
    \[f: G \to H, \quad u \mapsto u'.\]

    This is well-defined: since $G$ has no isolated points, for every $u \in G$ there is some $v \in G$ such that $(u,v) \in E$ or $(v,u) \in E$. Without loss of generality, we may assume the former and thus $u = \phi_{(u,v)}(\alpha)$. Hence $g(u)=g \circ \phi_{(u,v)}(\alpha)=\phi_{(u',v')}(\alpha)=u'$. Moreover, since $(u,v) \in E$ implies $(f(u),f(v))\in S$ it follows that $f:(G,E)\to(H,S)$ is a graph homomorphism. We claim that $(f \star M) = g$.

    Indeed, let $(u,a) \in G \times A \subseteq G \star M$. By construction, there must be a $v \in G$ and an edge $(u,v) \in E$. Hence, 
    \[ g(u,a)=g \circ \phi_{(u,v)}(a)=\phi_{(f(u),f(v))}(a)=(f(u),a).\]
    Similarly if $(v,b) \in G\times B\subseteq G\star M$, then $g(v,b)=(f(v),b)$. If $(u,v,c) \in (E \times (M\setminus (\{a,b\}\cup A\cup B \cup P)))$, then 
    \[ g(u,v,c)=g \circ \phi_{(u,v)}(c)=\phi_{(f(u),f(v))}(c)=(f(u),f(v),c). \]
    Finally, if $p \in P \subseteq G\star M$ then for any $(u,v) \in E$ it holds that 
    \[ g(p)=g \circ \phi_{(u,v)}(p)=\phi_{(f(u),f(v))}(p)=p,\]
    and therefore $g = f \star M$ as required. 
\end{proof}

Observe that condition (\ref{2}) implies that $M$ is rigid (and so in particular it has no isolated points), however the converse does not hold, e.g.\ $\L = \{<\}$ and $M$ is the total order on $n \in \omega$ elements. Furthermore, it implies that $P$ is fixed pointwise by all homomorphisms $G\star M \to H\star M$. Moreover, condition (\ref{3}) is clearly satisfied automatically whenever $\Cfrak$ is a full subcategory of $\StrL$. 

Notice that since $\star$ maps into the category of $\L$-structures, it is sensible to attempt to use an $\L$-structure $H\star M$ as an arrow. To do so under this framework, we ought to specify an $\L$-gadget structure on $G \star M$. Thankfully, this structure can be inherited from $G$ and $M$. In the remainder of this section we show how $\star$ can be extended to an operation on gadgets. Although this can also be achieved functorially, we do not explore this further here as it is not relevant to our proofs.  

\begin{definition}
    Let $(H,s,t)$ be a simple graph gadget, and $(M,\alpha,\beta,A,B,P)$ an $\L$-gadget. Define the $\L$-gadget
    \[ H \ostar M:=(G\star M,s,t,A',B',P), \text{ where}\]
    \[A'=\{(s,a):a \in A\},\quad B'=\{(t,b):b\in B\}.\]
\end{definition}

\begin{proposition}\label{stargadget}
    Let $G$ be a graph, $(H,s,t)$ be a simple graph gadget such that $s \in \pi_1[E(H)],t\in \pi_2[E(H)]$, and $M$ an $\L$-gadget. Then:
    \[ 
    (G \star H) \star M \iso G \star (H \ostar M).
    \]
\end{proposition}

\begin{proof}
Observe that if $G$ has $\kappa$ isolated points and $H$ has $\lambda$ isolated points, then both $(G \star H)\star M$ and $G \star (H \ostar M)$ have exactly $\kappa+|E(G)|\times \lambda$ isolated points that do not come from copies of $M$. Since there is a bijection between these, we may restrict to the subgraphs of $G$ and $H$ that do not contain isolated points. 

So, without loss of generality assume that $G$ and $H$ do not contain isolated points. Recall that $(x,y) \in E(G \star H)$ if and only if there are $(u,v)\in E(G)$ and $(z,w) \in E(H)$ such that 
    \[ \phi_{u,v}^{G,H}(z,w)=(x,y).\]
    Hence for all $(u,v) \in E(G)$ and $(z,w) \in E(H)$ consider the injective strong homomorphisms
    \[ \psi^{u,v}_{z,w}=\phi_{\phi_{u,v}(z,w)}^{G\star H,M}.\]

    Consider 
    \[ f_{u,v,z,w}=\phi_{u,v}^{G,H\ostar M}\circ \phi_{z,w}^{H,M}\circ(\psi^{u,v}_{z,w})^{-1}:\psi^{u,v}_{z,w}[M] \to G\star (H\ostar M),\]
    
    \[ f = \bigcup_{(u,v)\in E(G)}\bigcup_{(z,w)\in E(H)} f_{u,v,z,w}:(G\star H)\star M \to G\star (H\ostar M).\]

    So, for every $(u,v)\in E(G)$ and every $(z,w) \in E(H)$ we have a diagram:

    \begin{figure}[H]
  \centering\small

\begin{tikzcd}
M \arrow[dd, "{\phi_{z,w}^{H,M}}"'] \arrow[rr, "{\psi^{u,v}_{z,w}}"] &  & (G\star H)\star M \arrow[dd, "f"] \\
                                                                                &  &                                   \\
H\ostar M \arrow[rr, "{\phi_{u,v}^{G,H\ostar M}}"']                             &  & G \star (H \ostar M)             
\end{tikzcd}

\end{figure}

    We argue that $f$ is an isomorphism of $\L$-structures. Indeed, every element of $(G \star H)\star M$ is in the image of $\psi^{u,v}_{z,w}$ for some $(u,v)\in E(G)$ and $(z,w) \in E(H)$, since $G$ and $H$ do not contain isolated points. Therefore $f$ is defined on all of $(G\star H)\star M$. Furthermore, it is a well defined map since
    \[\psi^{u,v}_{z,w}[M] \cap \psi^{u',v'}_{z',w'}[M] =  \]
    \begin{equation*}\label{intersect2}
    \begin{cases}
       \psi^{u,v}_{z,w}[M], & \text{if } (u,v)=(u',v')\land (z,w)=(z',w');  \\
       
       \{\phi_{u,v}(z)\}\cup\{(\phi_{u,v}(z),a):a \in A\} \cup P, &\text{if } \phi_{u,v}(z)=\phi_{u',v'}(z') \land  \phi_{u,v}(w)\neq \phi_{u',v'}(w'); \\
       
       \{\phi_{u,v}(w)\}\cup\{(\phi_{u,v}(w),b):b \in B\}\cup P, &\text{if } \phi_{u,v}(z)\neq\phi_{u',v'}(z')  \land  \phi_{u,v}(w)= \phi_{u',v'}(w'); \\
       
       P, &\text{otherwise.}
     \end{cases}
    \end{equation*}

    and the maps $f_{u,v,z,w}$ agree on these sets. Finally, every element of $H \ostar M$ lies in the image of $\phi^{H,M}_{z,w}$ for some $(z,w) \in E(H)$ since $H$ has no isolated points, and moreover every element of $G \star (H \ostar M)$ lies in the image of $\phi^{G,H\ostar M}_{u,v}$ for some $(u,v)\in E(G)$ since $G$ has no isolated points. It follows that $f$ is surjective, and since all of $\psi^{u,v}_{z,w},\phi^{H,M}_{z,w},\phi^{G,H\ostar M}_{u,v}$ are injective strong homomorphisms, $f$ is an isomorphism. 
\end{proof}

As an example, observe that the directed $r$-subdivision of a graph $H$ is isomorphic to $H \star P_r$ where $(P_r,\alpha,\beta)$ is a simple graph gadget with $P_r$ a directed graph path of length $r+1$ from $\alpha$ to $\beta$. Hence, \Cref{lem:ostar} implies that for all graphs $G$ and simple graph gadgets $(H,s,t)$ 
\[G \star H^{(r)} \iso (G \star H)^{(r)}.\] 

In the next section we establish that the category of well-founded graphs is algebraically universal using a simple gadget construction. This is important for our proof of the partial converse to \Cref{cor:swd}, established in \Cref{sec:monotone}. 

\section{The category of well-founded graphs}\label{sec:wellfdd}

Consider the category $\Ord$ of ordinals with homomorphisms, i.e.\ order-preserving maps. It is easy to observe that $\Ord$ is not $\kappa$-universal for any infinite cardinal $\kappa$. Indeed, for any $\alpha,\beta \in \obj(\Ord)$, either there is a homomorphism $\alpha \to \beta$ or a homomorphism $\beta \to \alpha$. On the other hand there are plenty of finite graphs which do not map homomorphically into one another. Here, we argue that although $\Ord$ is not algebraically universal, its monotone closure is. This precisely corresponds to the category of well-founded graphs. Recall that a binary relation $R \subseteq X \times X$ is \emph{well-founded} on $X$ if for all non-empty $S \subseteq X$ there is some $s \in S$ such that no $p \in S$ satisfies $p R s$. Recall the following standard lemma. 

\begin{lemma}
    The transitive closure of a well-founded binary relation $R\subseteq X \times X$ is a well-founded strict partial order $(X,<)$. 
\end{lemma}

\begin{proof}
    We first argue that any well-founded structure $(X,R)$ is an acyclic directed graph. The definition trivially implies that no $u \in X$ satisfies $u R u$, while if $u R v$ then taking $S=\{u,v\}$ in the definition gives that $s = u$ and so $u \not R v$. Hence $(X,R)$ is a directed graph. Now if $s_1 R s_2 \dots s_n R s_1$ is a cycle in $X$, then taking $S=\{s_1,\dots,s_n\}$ there must be some $s \in S$ such that no $s_i$ satisfies $s_i R s$; contradiction. It follows that $(X,R)$ is acyclic.

    Let $<$ be the transitive closure of $R$. Since $(X,R)$ is an acyclic directed graph, it follows by standard facts that $(X,<)$ is a strict partial order. We argue that $<$ is still well-founded on $X$. Suppose that there is some $S \subseteq X$ such that for all $s \in S$ there exists some $p \in S$ with $p<s$. For every such pair we may associate a finite sequence $p=u_0 R u_1 R \dots R u_{n_s}=s$ of elements from $X$. Letting $S'$ be the subset of $X$ containing the elements of all such sequences, we see that $S'$ witnesses that $(X,R)$ is not well-founded. Consequently, no such $S\subseteq X$ exists and therefore $(X,<)$ is well-founded.  
\end{proof}

By the above, given a well-founded relation $R\subseteq X \times X$ we may view the pair $(X,R)$ as an acyclic directed graph. Hence, we call any acyclic directed graph $(G,E)$ such that $E$ is well-founded in $G$ a \emph{well-founded graph}. We write $\W\Gra$ for the category of well-founded graphs with homomorphisms. The following lemma illustrates that the category of well-founded graphs is precisely the full monotone closure of $\Ord$.

\begin{lemma}
    The following are equivalent for a graph $(X,R)$:
    \begin{enumerate}
        \item $R$ is well-founded on $X$;
        \item There is an ordinal $\alpha<|X|^+$ and an injective homomorphism $f:(X,R)\to (\alpha,\in)$;
        \item $(X,R^*)$ does not contain infinite paths, where $R^*=\{(x,y):(y,x) \in R\}$
        
    \end{enumerate}
\end{lemma}

\begin{proof}
    $(1)\implies(2)$: Suppose that $(X,R)$ is well-founded. By the previous lemma, its transitive closure $(X,<)$ is a well-founded strict partial order. By Zorn's Lemma, $(X,<)$ extends to a strict total order $(X,\prec)$. This is well-ordered. Indeed, for any $S \subseteq X$ there is some $s \in S$ such that no $p \in S$ satisfies $p<s$, and therefore $s \preceq p$. Hence, $(X,\prec)$ is isomorphic to a unique ordinal $\alpha < |X|^+$. The composition of this isomorphism with the natural injective homomorphism $(X,<) \to (X,\prec)$ gives an injective homomorphism $f:(X,<)\to (\alpha,\in)$. Composing once more with the injective homomorphism $(X,R) \to (X,<)$, we obtain an injective homomorphism $(X,R) \to (\alpha,\in)$. 

    $(2)\implies (3)$: If $f:(X,R)\to(\alpha,\in)$ is an injective homomorphism and $(X,R^*)$ contains an infinite path $u_0 R^* u_1 R^* \dots$, then $f(u_0)\ni f(u_1)\ni\dots$ is an infinite strictly decreasing sequence in $\alpha$; contradiction. 

    $(3)\implies (1)$: If $(X,R)$ is not well-founded then there is some $S \subseteq X$ such that for every $s \in S$ there is some $p \in S$ with $p R s$. We may therefore inductively built a path $s_0 R^* s_1 R^* s_2 R^* \dots$, and hence $(X,R^*)$ contains an infinite path. 
\end{proof}

 We hence establish that $\W\Gra$ is algebraically universal, and in fact, $\kappa$-universal for all cardinals $\kappa$. This is achieved using the same construction as the one in Chapter IV, Theorem 3.5 of \cite{pultr}. 

 \begin{definition}\label{hcal}
     Henceforth, fix the graph
    
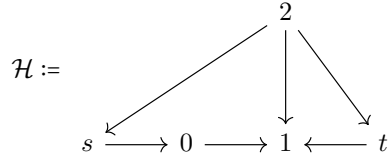
\begin{figure}[H]
  \centering\small

  $\Hcal:=$
\begin{tikzcd}
                 &             & 2 \arrow[lldd] \arrow[dd] \arrow[rdd] &                 \\
                 &             &                                                   &                 \\
s \arrow[r] & 0 \arrow[r] & 1                                                 & t \arrow[l]
\end{tikzcd}

\end{figure}

and the simple graph gadget $(\Hcal,s,t)$. Clearly, $\Hcal$ is well-founded, and so is any graph $G \star \Hcal$.
 \end{definition}
 
 This will come up in several arguments, precisely because it satisfies condition \ref{2} of \Cref{th:system} in a strong sense. 

\begin{lemma}\label{lem:Hcal}
    For all $r \in \omega$ and graphs $G$, the only homomorphisms from $\Hcal^{(r)} \to G \star \Hcal^{(r)}$ are the maps $\phi_{u,v}^{\Hcal^{(r)},G}$ for $(u,v)\in E(G)$.
\end{lemma}

\begin{proof}
    Let $f:\Hcal^{(r)} \to G \star \Hcal^{(r)}$ be a homomorphism. Since $\Hcal$ is connected and $2$ is the only element of $\Hcal$ without incoming edges, it follows that $f(2)=(2,u,v)$ for some $(u,v) \in E(G)$. Likewise, since $1$ is the only element of $\Hcal$ without outgoing edges it follows that $f(1)=(1,u',v')$ for some $(u',v') \in E(G)$. Observe that $G \star H$ has no loops or multiple edges, and so the graph path of length $r+1$ from $2$ to $1$ must map under $f$ to a graph path of length $r+1$ from $(2,u,v)$ to $(1,u',v')$. Hence, it follows that $(u,v)=(u',v')$. Similarly, the graph path of length $2(r+1)$ from $2$ to $1$ that passes through $t$ must necessarily be mapped under $f$ to a graph path of length $2(r+1)$ from $(2,u,v)$ to $(1,u,v)$. The only such graph path in $G \star \Hcal$ is the one passing through $(t,v)$, and so in particular $f(t)=(t,v)$. Finally, the graph path of length $3(r+1)$ from $2$ to $1$ passing through $s$ and $0$ must also be mapped to a graph path of length $3(r+1)$ from $(2,u,v)$ to $(1,u,v)$, and so by similar reasoning it follows that $f(s)=(s,u)$ and $f(0)=(0,u,v)$. Since all subdivision points of $\Hcal^{(r)}$ are mapped to the corresponding subdivision points in $G \star \Hcal^{(r)}$, it follows that $f = \phi_{u,v}^{\Hcal^{(r)},G}$. 
\end{proof}

As a consequence of the above, $\W\Gra$ is algebraically universal. 

\begin{corollary}
    The category of well-founded graphs is algebraically universal. Moreover, it is $\kappa$-universal for all infinite cardinals $\kappa$. 
\end{corollary}

\begin{proof}
  Consider the simple graph system $(\Hcal,s,t)$ from \Cref{hcal}. By standard facts, the category $\mathbf{Con}\Gra$ of connected graphs is algebraically universal and $\kappa$-universal for all infinite $\kappa$ (see \cite{pultr}). We thus argue that the functor $\Phi: \mathbf{Con}\Gra \to \W\Gra$ given by $G \mapsto G \star \Hcal$ is a full embedding. By \Cref{th:system}, it suffices to check that for any connected graph $G$, the only homomorphisms $\Hcal \to G \star \Hcal$ are the ones of the form $\phi_{u,v}^{H,G}$ for $(u,v) \in E(G)$. This follows by \Cref{lem:Hcal}. Finally, since $\Hcal$ is finite, this implies that $\W\Gra$ is $\kappa$-universal for all infinite $\kappa$. 
\end{proof}

\section{Directed gadgets and paths}\label{sec:directed}

Although somewhere density is a necessary condition for algebraic universality, it is evidently not sufficient as illustrated by the category of bipartite undirected graphs. Even in the presence of monotonicity, we can see that the full category of bipartite graphs with directed edges only from the left part to the right part has somewhere dense Gaifman class, is monotone, but is not algebraically universal. Nonetheless, after assigning a suitable orientation to the graphs in our category we may obtain an algebraically universal category of directed graphs. Here, we provide an analogue of this idea for relational structures, by introducing \emph{directed $\L$-structures}. These come with a natural functor into the category of directed graphs. 

\begin{definition}\label{defdirected}
    Let $M$ be a $\L$-structure, and write $V=\bigcup_{R \in \L}(\pi_0[R^M]\cup \pi_1[R^M])\subseteq M$ for the set of elements of $M$ appearing in the first two coordinates of the interpretation of any relation symbol in $M$. We say that $M$ is \emph{directed} if for all $u, v \in V$ there is at most one $R \in \L$ and at most one $\bar a \in R^M$ with $(u,v)=(\pi_0(\bar a),\pi_1(\bar a))$ or $(u,v)=(\pi_1(\bar a),\pi_0(\bar a))$, and moreover there is no $R \in \L$ and $\bar a \in R^M$ with $\pi_0(\bar a)=\pi_1(\bar a)$. 

    Given a directed $\L$-structure we define its \emph{arc graph}, denoted by $\Arc(M)$, as the directed graph with vertex set $V$ and edge set
    \[ E = \{(u,v) \in V^2 : \exists R\in \L, \bar a \in R^M \text{ such that } \pi_0(\bar a)=u \text{ and } \pi_1(\bar a)=v\}.\]
    \end{definition}
   
We write $\overrightarrow{\Str}(\L)$ for the category for directed $\L$-structures with homomorphisms. We call any $\L$-gadget $(M,\alpha,\beta,A,B,P)$ such that $M$ is directed a \emph{directed $\L$-gadget}, and write $\overrightarrow{\Gdg}(\L)$ for the category of directed $\L$-gadgets with $\L$-gadget homomorphisms.

Note that the choice of $\pi_0$ and $\pi_1$ is arbitrary, and any two distinct projections would give a suitable definition. However, there is a slight preference for this definition as a directed $\{E\}$-structure is a directed graph in the standard sense. We therefore call a directed $\{E\}$-gadget a \emph{directed graph gadget}, and write $\overrightarrow{\Gra\Gdg}$ for the category of directed graph gadgets. Moreover, it is easy to see that for any directed graph $G$, $\Arc(G)$ is the induced subgraph of $G$ on the vertices that are not isolated. Restricting to $\L$-structures without isolated points ensures that $\Arc$ is a faithful functor. 

\begin{lemma}\label{arc}
    $\Arc: \overrightarrow{\Str}(\L)^{-} \to \overrightarrow{\Gra}^-$ is a faithful functor from the category of directed $\L$-structures without isolated points and $\L$-homomorphisms to the category of directed graphs without isolated points and graph homomorphisms.
\end{lemma}

\begin{proof}
    Fix two directed $\L$-structures $M,N$ and a homomorphism $f: M \to N$. If $v \in V_M$, then by definition there is some $R_v \in \L$ and $\bar a_v \in R^M$ such that $v=\pi_i(\bar a_v)$ for $i \in \{0,1\}$. It follows that $f(\bar a_v) \in R^N$ and moreover $f(v)=\pi_i \circ f(\bar a_v)$, so $f(v) \in V_N$. Hence, the restriction of $f$ to $V_M$ gives a map 
    \[ \Arc(f) : \Arc(M) \to \Arc(N) \]
    We claim that this is a graph homomorphism. Indeed, if $(u,v) \in E(\Arc(M))$, then there is a (unique) $R \in \L$ and $\bar a \in R^M$ such that $u=\pi_0(\bar a), v=\pi_1(\bar a)$. Again, it follows that $f(\bar a) \in R^N$ and $f(u) =\pi_0\circ f (\bar a),f(v) = \pi_1\circ f(\bar a)$, so $(f(u),f(v))$ is an edge in $\Arc(N)$.
    
    Finally, we claim that $\Arc$ is injective on morphisms. Suppose that $M,N$ are directed without isolated points and $f,g:M \to N$ are two distinct homomorphisms, so there is some $c \in M$ such that $f(c)\neq g(c)$. If $c \in V_M$, then $\Arc(f)\neq \Arc(g)$. If on the other hand $c \notin V_M$, then there is some $R \in L$ and $\bar a \in R^M$ such that $c=\pi_i(\bar a)$ for some $i>1$. It follows that $f(\bar a)\neq g(\bar a)$. Letting $u=\pi_0(\bar a), v=\pi_1(\bar a)$ we see that $(f(u),f(v))\neq(g(u),g(v))$ since $N$ is directed. Since $u,v \in V_M$, it follows that $\Arc(f)\neq\Arc(g)$.
\end{proof} 

\begin{definition}
    We call a directed $\L$-gadget $(M,\alpha,\beta,A,B,P)$ with $\alpha,\beta \in \Arc(M)$ and $\Arc(M)\cap(A \cup B\cup P)=\emptyset$ an \emph{$\L$-system}. When $\L=\{E\}$, we call it a \emph{graph system}. We write $\Syst(\L)$ for the category of $\L$-systems with $\L$-gadget homomorphisms. We also extend $\Arc$ to a map on $\L$-gadgets, such that
    \[\Arc: (M,\alpha,\beta,A,B,P) \mapsto (\Arc(M),\alpha,\beta,\emptyset,\emptyset,\emptyset)\] 
\end{definition}

Although we omit a proof because it is not relevant to our purposes, an adaptation of \Cref{arc} shows that that $\Arc$ extends to a faithful functor on $\L$-systems without isolated points. Next, we show that $G\star M$ is a directed $\L$-structure whenever $G$ is a directed graph and $M$ is an $\L$-system.

\begin{lemma}
    Let $(M,\alpha,\beta,A,B,P)$ be a directed $\L$-gadget and $G$ a graph. Then for all $(u,v) \in E(G)$, $\phi_{(u,v)}^{G,M}:M \to G \star M$ restricts to a graph isomorphism 
    \[\phi_{(u,v)}^{G,M}:\Arc(M) \to \Arc(\phi_{(u,v)}[M]).\]
\end{lemma}

\begin{proof}
    Trivially, since $\phi_{(u,v)}$ is an injective strong homomorphism $\phi_{(u,v)}[M]\iso M$ is directed and they have isomorphic arc graphs.
\end{proof}

\begin{proposition}\label{directed}
    Let $G$ be a directed graph, and $(M,\alpha,\beta,A,B,P)$ an $\L$-system. Then $G \star M$ is directed. 
    
\end{proposition}

\begin{proof}
     We claim that $V_{G\star M}=\bigcup_{R \in \L}(\pi_0[R^{G\star M}]\cup \pi_1[R^{G\star M}])$ satisfies the condition of Definition \ref{defdirected}. Indeed observe that $m \in V_{G \star M}$ if and only if there is and edge $(u,v) \in E(G), R \in \L$ and $\bar c \in R^M$ such that $\pi_i(\phi_{(u,v)}(\bar c))=m$, for some $i \in \{0,1\}$. Thus:

    \[ V_{G\star M} = \bigcup_{(u,v)\in E(G)} \Arc(\phi_{(u,v)}[M]).\]

    Now, notice that since $\alpha, \beta \in \Arc(M)$ it follows that $u,v \in \Arc(\phi_{(u,v)}[M])$. In addition, since $\Arc(M)$ does not intersect $A,B,P$ and $G$ is directed it holds that:

    \begin{equation*}\label{intersect}\Arc(\phi_{(u,v)}[M]) \cap \Arc(\phi_{(u',v')}[M]) =  
    \begin{cases}
       \Arc(\phi_{(u,v)}[M]), &\quad \text{if } (u,v)=(u',v');  \\
       \{u\}, &\quad \text{if } u=u' \land v\neq v' \\
       \{v\}, &\quad \text{if } u \neq u' \land v=v' \\
       \emptyset, &\quad \text{otherwise.}
    \end{cases}
    \end{equation*}

    Hence, $x\neq y$ implies that $\{x,y\}\not\subseteq\Arc(\phi_{(u,v)}[M]) \cap \Arc(\phi_{(u',v')}[M])$ for $(u,v)\neq (u',v')$. Consequently, given that each $\Arc(\phi_{(u,v)}[M])$ is directed, it follows that for all $x \neq y$ from $V_{G \star M}$ there is at most one $R \in \L$ and at most one $\bar a \in R^{G\star M}$ with $(u,v)=(\pi_0(\bar a),\pi_1(\bar a))$ or $(u,v)=(\pi_1(\bar a),\pi_0(\bar a))$. Hence, $G \star M$ is directed. 
\end{proof}

Hence, since $G\star M$ is directed, we may consider its arc graph. As it turns out, this is equal to the directed graph $G \star \Arc(M)$.

\begin{proposition}\label{arcstar}
 For a directed graph $G$ and an $\L$-system $(M,\alpha,\beta,A,B,P)$  
\[\Arc(G\star M) = G\star \Arc(M).\]

Moreover, for all $(u,v) \in E(G)$: 
\[ \Arc(\phi_{u,v}^{M,G}) = \phi_{u,v}^{\Arc(M),G}.\]
\end{proposition}

\begin{proof}
    By Lemma \ref{directed}, $G\star M$ is directed so $\Arc(G\star M)$ is well defined and equal to $\bigcup_{(u,v)\in E(G)} \Arc(\phi_{(u,v)}[M])$. Since $M$ is a system and $\Arc(\phi_{(u,v)}[M])\iso\Arc(M)$, this implies that
    \[ \Arc(G\star M) = G \sqcup \emptyset \sqcup \emptyset \sqcup \{(u,v,m): (u,v)\in E(G), m \in \Arc(M)\} \sqcup \emptyset.\]
    
    It follows that $\Arc(G\star M)$ and $G \star \Arc(M)$ are equal as sets. Moreover, $(x,y)$ is an edge in $\Arc(G\star M)$ if and only if there is an edge $(u,v) \in E(G)$ and an edge $(w,z) \in E(\Arc(M))$ such that $\phi_{(u,v)}(w,z)=(x,y)$, which by definition occurs if and only if $(x,y)$ is an edge in $G\star \Arc(M)$.

    For the second claim, recall that by definition $\Arc(\phi_{u,v}^{M,G})$ is the restriction of $\phi_{u,v}^{M,G}$ on $\Arc(M)$. Consequently, since $\Arc(M)$ is simple it follows that
    \[ \Arc(\phi_{u,v}^{M,G})(x) = 
    \begin{cases}
        u &\quad \text{if }x=\alpha; \\
        v &\quad \text{if }x=\beta; \\ 
        (x,u,v) &\quad \text{otherwise.}
    \end{cases}
    \]
    which is therefore equal to $\phi_{u,v}^{\Arc(M),G}$ by definition. 
\end{proof}

Put together, all the above imply the following lemma which is crucially used in \Cref{sec:monotone}.

\begin{lemma}\label{lem:ostar}
    Let $G$ be a directed graph, $(H,s,t)$ a simple graph system, and $(M,\alpha,\beta,A,B,P)$ and $\L$-system. Then $H \ostar M$ is an $\L$-system, and moreover
    \[
    \Arc(G\star (H\ostar M)) \iso (G\star H)\star \Arc(M). 
    \]
\end{lemma}

\begin{proof}
    By Lemma \ref{directed}, $H\star M$ is a directed $\L$-structure with $\Arc(G\star M) = \bigcup_{(u,v)\in E(G)}\Arc(\phi_{u,v}[M])$. Since $\alpha,\beta \in \Arc(M)$, it follows that $s,t \in \Arc(G \star M)$, while $\Arc(M)\cap(A\cup B\cup P)=\emptyset$ implies that 
    \[ \Arc(G\star M)\cap (\{(s,a):a \in A\}\cup \{(t,b):b\in B\}\cup P)=\emptyset.\]
    So $H \ostar M$ is in fact an $\L$-system. 
    By Proposition \ref{stargadget}, it follows that 
    \[G \star (H\ostar M) \iso (G\star H)\star M,\]
    and so by Theorem \ref{arcstar}
    \[ \Arc(G\star (H\star M))\iso \Arc((G\star H)\star M) \iso (G\star H) \star \Arc(M).\]
\end{proof}

In the remainder of this section we introduce relational paths. As later shown in \Cref{th:main}, these are precisely the structures that we use as arrows to define a functor into a monotone somewhere dense category. To these, we may naturally assign a ``direction'' and obtain an equivalent directed $\L$-structure. We note that a variant of relational paths also appears in \cite{monotoneNIP}, but our definition here is less restrictive. 

\begin{definition}\label{def:path}
By an \emph{$\L$-path of length $n$} we mean an $\L$-structure $M$ together with an injective function $p:n+1 \to M$ such that there is a sequence of tuples $\bar e_1, \dots, \bar e_n$ from $M$ satisfying: 
\begin{itemize}
    \item $M=\bigcup_{i \in [n]}\bar e_i$;
    \item $p(0) \in \bar e_1\setminus \bar e_2, p(n) \in \bar e_n\setminus \bar e_{n-1}$, and $p(i)\in \bar e_i\cap e_{i+1}$ for all $i \in [n-1]$;
    \item $R^M_i(\bar e_i)$, for some unique relation symbol $R_i \in \L$;
    \item $R^M(\bar a) \implies \bar a = \bar e_i$ for some $i \in [n]$, for all relation symbols $R \in \L$ and all tuples $\bar a \in M$;
\end{itemize}
We refer to the tuples $\bar e_i$ as the \emph{steps} of the path, and to the elements $p(i)$ for $i \in [n-1]$ as the \emph{joints} for $M$. We also refer to the map $p:n+1 \to M$ as the \emph{path function} of the path.
\end{definition}

Note that technically, no undirected graph $G$ can be an $\L$-path under the above definition. Indeed, the last condition ensures that $E(G)$ cannot be symmetric as no permutation of a tuple appearing in a relation $R$ can appear in any other relation from $\L$. To avoid confusion, we always refer to paths in the standard graph-theoretic sense as \emph{graph paths}. 

\begin{lemma}\label{lem:permutation}
    Let $(M,p)$ be an $\L$-path of length $n$. Then there is a a directed $\L$-structure $\tilde{M}$ that is permutation equivalent to $M$, $(\tilde{M},p)$ is an $\L$-path, and $\Arc(\tilde{M})$ is the directed graph path $p(0), p(1),\dots,p(n)$. 
\end{lemma}

\begin{proof}
    Let $(M,p)$ be as above and let $\bar e_i$ be its steps. For each $i \in n+1$, fix some permutation $\sigma_i \in S_{|\bar e_i|}$ such that $\pi_0(\sigma_i(\bar e_i))=p(i-1)$ and $\pi_1(\sigma_i(\bar e_i))=p(i)$. Let $\tilde{M}$ be the $\L$-structure on the same domain as $M$, satisfying for all $R \in \L$ and $\bar a \in \tilde{M}$:
    \[ \bar a \in R^{\tilde M} \iff \bar a=\sigma_i(\bar e_i) \text{ for some } i \in [n] \text{ and } \bar e_i \in R^M.\]
    It is clear that $(\tilde{M},p)$ is still a path of length $n$ with steps $\sigma_i(\bar e_i)$ for $i \in [n]$. Moreover, $M$ and $\tilde M$ are by definition permutation equivalent. We argue that $\tilde{M}$ is a directed $\L$-structure. Indeed, by construction $\bigcup_{R \in \L}(\pi_0[R^M]\cup \pi_1[R^M])=\{p(i):i \in n+1\}$. Furthermore, it follows that for any two $i\neq j$ from $n$, there is at most one step $\bar e$ of $\tilde{M}$ such that $(p(i),p(j)=(\pi_0(\bar e),\pi_1(\bar e))$ or $(p(i),p(j))=(\pi_1(\bar e),\pi_0(\bar e))$. Since the only tuples from $\tilde{M}$ appearing in a relation $R^{\tilde{M}}$ are its steps, this implies that $\tilde{M}$ is directed. Finally, it is easy to see that 
    \[ E(\Arc(\tilde{M}))=\{(p(i),p(i+1)):i \in n-1\}\]
    as required.
\end{proof}

Given an $\L$-structure and a graph path in $\Gaif(M)$, we may produce an $\L$-path that describes a  ``type'' for this graph path. This idea is captured by the following definition.

\begin{definition}[Path type]\label{def:pathtype}
    Let $N$ be an $\L$-structure, and $u_0,\dots,u_n$ a graph path in $\Gaif(N)$. For every $i \in n$ we may associate a relation symbol $R_i \in \L$, elements $v_{i,1},\dots,v_{i,\ar(R_i)}$ of $N$, and a permutation $\sigma_i \in S_{\ar(R_i)}$ such that $\sigma_i(u_i,u_{i+1},\bar v_i) \in R_i^N$. Letting $M=\{u_i:i \in n+1\}\cup\{v_{i,j}: i \in n, j \in [\ar(R_i)]\}$, we define an $\L$-structure on $M$ such that for all $R \in \L$ and tuples $\bar a$ from $T$:
    \[ \bar a \in R^M \iff \bar a = \sigma_i(u_i,u_{i+1},\bar v_i) \text{ and }  R=R_i \text{ for some } i \in n .\]
    Letting $p: n+1 \to M$ be the injective map $i \mapsto u_i$, it is easy to see that $(M,p)$ is a path of length $n$. We call the tuple $(M,p)$ a \emph{path type} in $M$ for the graph path $u_0,\dots, u_n$.  
\end{definition}

Observe that that whenever $S=(u_0,\dots,u_n)$ is a graph path in $\Gaif(N)$, then there is a path type $(M,p)$ for $S$ in $N$ and that the identity map $M \to N$ is an injective homomorphism. Clearly, the path type is not uniquely determined by $S$, as for the same graph path $u_1,\dots,u_n$ in $\Gaif(N)$ we can possibly obtain different sequences of relations $R_i,\dots R_{i-1}$ and permutations $\sigma_1,\dots,\sigma_{i-1}$ to satisfy Definition \ref{def:pathtype}.

\section{Monotone algebraically universal categories}\label{sec:monotone}

Here we establish the partial converse to \Cref{cor:swd}. As previously illustrated, we cannot hope for a full converse, not even in the presence of monotonicity. This is overcome by working with an orientation $\tilde \Cfrak$ of $\Cfrak$. By picking an appropriate such orientation we can make sure that the structures in $\tilde \Cfrak$ are directed, and thus, applying the $\Arc$ functor allows us to reduce the argument to the graph case which is simpler to handle.

Once again, we have opted for a relativised proof to illustrate the exact levels where our theorem holds. Much like with \Cref{th:smwdense}, it shall be clear that the proof of the relativised statement subsumes the proof of the absolute one. We note that the regularity of $\kappa$ is only used to ensure that any language $\L$ of size $<\kappa$ works. For finite languages, the claim also holds at the level of strong limit cardinals. We begin with the following proposition, which is essentially a strengthening of Theorem 4.3 in \cite{monotoneNIP}. 

\begin{proposition}\label{th:main}
Let $\kappa$ be $\omega$ or an inaccessible cardinal and $\Cfrak$ a monotone subcategory of $\StrL$ such that $\Gaif(\Cfrak)$ is $\kappa$-somewhere dense and $|\L|<\kappa$. Then there is a finite $\L$-gadget $(M,\alpha,\beta,A,B,P)$ such that $G \star M \in \obj(\Cfrak)$ for every well-founded graph $G$ of size $<\kappa$. Moreover, there is $r \in \omega$ and a map $p:r+1 \to M$ such that:
\begin{enumerate}
    \item $(M,p)$ is an $\L$-path;
    \item\label{cond:2} $p(0)=\alpha$ and $p(r)=\beta$;
    \item\label{cond:3} $p(i) \notin (A\cup B\cup P)$ for all $i \in r+1$.
\end{enumerate}
\end{proposition}
\begin{proof}
If $\Gaif(\C)$ is $\kappa$-somewhere dense, then there exists $r \in \omega$ such that for all cardinals $\lambda < \kappa$ there is some $N_\lambda \in \obj(\Cfrak)$ with $K_\lambda^r \leq \Gaif(N_\lambda)$. Label the elements of $N_\lambda$ corresponding to the native vertices of $K^r_\lambda$ by $(a^\lambda_i)_{i<\lambda}$. For every $i < j$ from $\lambda$ let $S^\lambda_{i,j}$ be the graph path in $\Gaif(N_\lambda)$ corresponding to the $r$-subdivision of the edge $(i,j)$ from $K_\lambda$, directed from $a_i$ to $a_j$. Let $q <\kappa$ be the maximum arity of a relation symbol $R \in \L$. Observe that there are at most $p=(|\L|\cdot q!)^{r+1}\cdot2^{r^2}$ path types up two isomorphism for each graph path $S^\lambda_{i,j}$. By \Cref{cor:erdos} we may find for each $\lambda$ some $\Sigma_\lambda \subseteq \Rcal(2,\lambda,p)$ of size $\lambda$ such that $S^{\Rcal(2,\lambda,p)}_{i,j}$ have the same path type for all $i<j$ from $\Sigma_\lambda$. By passing to a subsequence of $(N_\lambda)_{\lambda < \kappa}$ and relabelling indices, we may therefore assume that all the $S^\lambda_{i,j}$ have the same path type up to isomorphism. Let this be $(M_\lambda,p_\lambda)$. Since the number of path types is $<\kappa$ and $\kappa$ is regular, \Cref{pigeonhole} implies that we may prune the sequence $(N_\lambda)_{\lambda < \kappa}$ once again to ensure that the same path type $(M,p)$ is obtained for all $\lambda < \kappa$. 

It follows by definition that for every $i<j \in \lambda$ there is an injective homomorphism $f^\lambda_{i,j}:M \to N_\lambda$ such that $f^\lambda_{i,j}(p(n))=S^\lambda_{i,j}(n)$ for all $n \in r+1$. Let $J=\{p(n):n \in r+1\}$ and consider $H = M \setminus J$ with $|H|=m$. By $m$ applications of \Cref{cor:erdos}, we may assume for all $x \in H$ and $i<j$, $k<l$ from $\lambda$ that whether $f^\lambda_{i,j}(x)=f^\lambda_{k,l}(x)$ depends on one of the four canonical cases from that theorem. Indeed, for every $\lambda < \kappa$ and each $x \in H$, define colourings $\chi_{\lambda,k}(i,j) = f^\lambda_{i,j}(x)$ of the two element subsets $\{i,j\}$ with $i<j$ of $\lambda$. It follows that $\Kcal^m(\lambda)<\kappa$ contains a subset $C_\lambda$ of order type $\lambda$ on which $\chi_{\Kcal^m(\lambda),x}$ is canonical for all $x \in H$. We may thus restrict the argument to the subsequence $(N_{\Kcal^m(\lambda)})_{\lambda < \kappa}$ and the maps $f^{\Kcal^m(\lambda)}_{i,j}$ for $i<j \in C_\lambda$, and relabel appropriately. For every $\lambda < \kappa$, after the relabelling, we have thus obtained a tuple $t_\lambda \in [4]^H$ such that $\chi_{\lambda,x}$ is canonical of type $t_\lambda(x)$. Since there are only finitely many such $t_\lambda$, by the pigeonhole principle we may consider a subsequence of $(N_\lambda)_{\lambda < \kappa}$ for which $t_\lambda$ is constant and equal to some $t\in [4]^H$, and relabel once more. 

We now proceed to turn the path $M$ into a suitable $\L$-gadget. Indeed, define the following subsets of $M$:
\begin{align*}
    P = \{ x \in H: t(x) = 1\}; \\
    A = \{ x \in H: t(x) = 2\}; \\
    B = \{x \in H: t(x) = 3 \}; \\
    H' = \{x \in H: t(x) = 4 \}.
\end{align*}
By the above, it follows that $f^\lambda_{i,j}(x)\neq f^\lambda_{k,l}(x)$ for all $(i,j)\neq (k,l)$ and $x \in J \cup H'$. Hence, by Lemma \ref{lem:compatible} there is a subset $I_\lambda$ of order type $\lambda$ such that, for all $i<j,k<l$ from $I_\lambda$ with $(i,j)\neq(k,l)$, $f^\lambda_{i,j}(x)\neq f^\lambda_{k,l}(y)$ for all $x,y \in J \cup H$ and $x\neq y$. It therefore follows that for all $i<j$, $k<l$ from $I_\lambda$,\ $f^\lambda_{i,j}(x)= f^\lambda_{k,l}(x)$ if, and only if $x=y$ and one of the following holds:
\begin{enumerate}
        \item $(i,j)=(k,l)$ and $x \in M$;
        \item $i = k, j \neq l$ and $x \in \{p(0)\}\cup A\cup P$;
        \item $i \neq k, j = l$ and $x \in \{p(r)\}\cup B \cup P$;
        \item $i \neq k, j \neq l$ and $x \in P$.

\end{enumerate}

Letting $\alpha=p(0)$ and $\beta \in p(r)$ and considering the $\L$-gadget $(M,\alpha,\beta,A,B,P)$, Lemma \ref{lem:injective} implies that there is an injective homomorphism from $\lambda \star M$ into $N_\lambda$. In particular, there is an injective homomorphism from $G \star M$ into $N_\lambda$ for every well-founded graph of size $<\lambda$. Since this holds for all infinite cardinals $\lambda<\kappa$ and $\Cfrak$ is monotone, it follows that $G \star M$ is in $\obj(\Cfrak)$ for every well-founded graph of size $<\kappa$. 
\end{proof} 

By considering an orientation of our category, we may ensure that the $\L$-gadget obtained above is in fact an $\L$-system whose arc graph is a path.

\begin{lemma}\label{lem:uniform}
    Let $\kappa$ be $\omega$ or an inaccessible cardinal and $\Cfrak$ a monotone subcategory of $\StrL$ such that $\Gaif(\C)$ is $\kappa$-somewhere dense and $|\L|<\kappa$. Then there is a full orientation $\tilde \Cfrak$ of $\Cfrak$ and a finite  $\L$-system $(M,\alpha,\beta,A,B,P)$ such that $G \star M \in \obj(\tilde \Cfrak)$ for every well-founded graph $G$ of size $<\kappa$. Moreover, $\Arc(M)$ is a directed graph path from $\alpha$ to $\beta$. 
\end{lemma}

\begin{proof}
    Let $(M,\alpha,\beta,A,B,P)$ be the $\L$-gadget obtained by \Cref{th:main} and $p:r+1 \to M$ its path function. By \Cref{lem:permutation}, it follows that there is a directed $\L$-structure $\tilde{M}$ on the same domain as $M$ which is permutation equivalent to $M$ and $\Arc(\tilde M)$ is the directed graph path $p(0),p(1),\dots,p(n)$. It follows by Conditions \ref{cond:2} and \ref{cond:3} of \Cref{th:main} that $\alpha=p(0),\beta=p(1) \in \Arc(\tilde M)$, while $\Arc(\tilde{M})\cap(A\cup B\cup P)=\emptyset$. Hence, the $\L$-gadget $(\tilde M,\alpha,\beta,A,B,P)$ is in fact an $\L$-system. Fixing a well-founded graph $G$ of size $<\kappa$, the fact that $G \star M \in \obj(\Cfrak)$ implies that we may choose a full orientation $\tilde \Cfrak$ of $\C$ such that $G \star \tilde M \in \obj(\tilde \Cfrak)$. 
\end{proof}

Putting all the above together, we finally obtain the following.

\begin{theorem}\label{main}
    Let $\kappa$ be $\omega$ or an inaccessible cardinal and $\Cfrak$ a monotone subcategory of $\StrL$ such that $\Gaif(\Cfrak)$ is $\kappa$-somewhere dense and $|\L|<\kappa$. Then there is a full orientation $\tilde \Cfrak$ of $\Cfrak$ such that $\tilde \Cfrak$ is $\kappa$-algebraically universal. 
\end{theorem}

\begin{proof}
    Let $\tilde \Cfrak$ be the full orientation of $\C$ and $(M,\alpha,\beta,A,B,P)$ the $\L$-system from \Cref{lem:uniform}. Fix the simple graph gadget $(\Hcal,s,t)$ from \Cref{hcal}. Observe that since $\Hcal$ is directed, $(\Hcal,s,t)$ is in fact a simple graph system. Define the $\L$-gadget $M_\Hcal = \Hcal \ostar M$. By \Cref{lem:ostar}, it follows that $M_\Hcal$ is an $\L$-system and moreover for every directed graph $G$:
\[ \Arc(G \star M_\Hcal) \iso (G\star H) \star \Arc(M).\]
By \Cref{lem:uniform}, $\Arc(M)$ is a directed graph path from $\alpha$ to $\beta$ and so 
\[ \Arc(G \star M_\Hcal) \iso (G \star H)^{(r)} \iso G \star H^{(r)}\]
for some $r \in \omega$. So, consider the functor 
    \[ \Phi:\W\Gra_{<\kappa} \to \tilde \Cfrak_{<\kappa}\]
    \[ G \mapsto G\star M_\Hcal, \quad f\mapsto f\star M_\Hcal.\]
 We use Theorem \ref{th:system} to argue that $\Phi$ is a full embedding. Indeed, fix $G \in \obj(\W\Gra_{<\kappa})$ and consider a homomorphism $f: M_\Hcal \to G \star M_\Hcal$ of $\L$-structures. This restricts to a homomorphism $\Arc(f):\Arc(M_\Hcal)\to \Arc(G\star M_\Hcal)$ of their corresponding arc graphs. We argue that $\Arc(f)=\Arc(\phi_{u,v}^{M_\Hcal,G})$ for some $(u,v) \in E(G)$. By \Cref{lem:Hcal}, the only homomorphisms from $\Hcal^{(r)}$ to $G \star H^{(r)}$ are the maps $\phi_{u,v}^{\Hcal^{(r)},G}$, and since $\Arc(M_\Hcal)\iso \Hcal^{(r)}$ and $G \star M_\Hcal \iso G \star H^{(r)}$ it follows that 
 \[ \Arc(f) = \phi_{u,v}^{\Arc(M_\Hcal),G} = \Arc(\phi_{u,v}^{M_\Hcal,G}),\]
the second equality holding by \Cref{arcstar}. Finally, by \Cref{arc} $\Arc$ is a faithful functor on directed $\L$-systems without isolated points, and since $M_\Hcal$ has no isolated points it follows that $f = \phi_{u,v}^{M_\Hcal,G}$. Since $\tilde \Cfrak$ is full and $\W\Gra_{<\kappa}$ is $\kappa$-algebraically universal, \Cref{th:system} implies that $\tilde \Cfrak$ is $\kappa$-algebraically universal.
 \end{proof}

 As with \Cref{th:smwdense}, we note that using the full Erd{\"o}s-Rado theorem (\Cref{erdos}) and \Cref{fact:pigeon} in place of Facts \ref{cor:erdos} and \ref{pigeonhole} respectively in \Cref{th:main}, we obtain the absolute version of \Cref{main}.

\begin{corollary}
    Let $\Cfrak$ be a monotone subcategory of $\StrL$ such that $\Gaif(\Cfrak)$ is totally somewhere dense. Then there is a full orientation $\tilde \Cfrak$ of $\Cfrak$ such that $\tilde \Cfrak$ is algebraically universal. 
\end{corollary}

\section{Some Questions}

We conclude with a few questions which currently seem out of reach, but are nonetheless theoretically interesting. As we saw, there are algebraically universal categories that are not $\omega$-universal, such as the category of posets or the category of semigroups. Both examples are not model-theoretically tame (e.g.\ they are not NIP), but a priori there could possibly be algebraically universal categories which are tame in $\L_{\omega,\omega}$. This motivates the first question. 

\begin{question}
    Is there an algebraically universal category of $\L$-structures that is NIP or NIP$_2$? Clearly, this will not be $\omega$-universal.  
\end{question}

Moreover, all ``natural'' examples of categories considered here that are $\kappa$-universal for some uncountable $\kappa$, are in fact $\lambda$-universal for all uncountable $\lambda$. Although it is possible to construct artificial examples where this fails, it is sensible to inquire whether this holds for all such natural examples, e.g.\ for all varieties of universal algebras with homomorphisms. This would be an analogue of Morley's theorem \cite{morley} for universal categories.

\begin{question}
    Let $\Vcal$ be a variety of universal algebras in a countable similarity type. If $\Vcal$ is $\kappa$-universal for some uncountable $\kappa$, then is $\Vcal$ $\lambda$-universal for all uncountable $\lambda$? 
\end{question}

We also state this in relational form.

\begin{question}
    Let $\L$ be a countable relational language, $T$ an $\L$-theory, and $\mathbf{Mod}(T)$ the category of all models of $T$ with $\L$-homomorphisms. If $\mathbf{Mod}(T)$ is $\kappa$-universal for some uncountable $\kappa$, then is $\mathbf{Mod}(T)$ $\lambda$-universal for all uncountable $\lambda$? 
\end{question}

\section*{Acknowledgments}
The author is thankful to Tomáš Jakl for exposing him to the area, and for reading and discussing earlier versions of this manuscript. 

\bibliographystyle{plain}
\bibliography{bibliography}

\end{document}